\documentclass[12pt]{amsart}
\usepackage{amssymb}
\usepackage{amsmath}
\usepackage{enumitem}
\usepackage{fullpage}
\usepackage{xcolor}
\usepackage[unicode,hypertexnames=false]{hyperref}
\hypersetup{colorlinks=true,linkcolor=blue!60!black,citecolor=blue!60!black,urlcolor=blue!60!black}
\allowdisplaybreaks
\newtheorem{thm}{Theorem}[section]
\newtheorem{prop}[thm]{Proposition}
\newtheorem{cor}[thm]{Corollary}
\newtheorem{lem}[thm]{Lemma}
\newtheorem{rem}[thm]{Remark}
\newtheorem{defn}[thm]{Definition}

\newtheorem{conj}[thm]{Conjecture}
\newtheorem{Exa}[thm]{Example}
\newtheorem{conv}[thm]{Convention}

\numberwithin{equation}{section}

\newcommand{\C}{\mathbb{C}}
\newcommand{\N}{\mathbb{N}}
\newcommand{\Z}{\mathbb{Z}}
\newcommand{\res}{\mbox{\rm Res}}
\newcommand{\Res}{\mbox{\rm Res}}
\newcommand{\Span}{\mbox{\rm Span}}
\newcommand{\End}{\mbox{\rm End}}
\newcommand{\Hom}{\mbox{\rm Hom}}
\newcommand{\wt}{\mathrm{wt} \, }

\newcommand{\Id}{\mbox{\rm Id}}
\begin{document}
\title[Generators and relations for $A_n(V)$]
{On generators and relations for \(O_n^\circ(V)\) in higher-level Zhu
	algebras, Part I}

\author{Nathan Vander Werf}
\address{Department of Mathematics and Statistics, University of Nebraska at Kearney, Kearney, NE 68849}
\email{vanderwerfnp@unk.edu}
\urladdr{\href{https://orcid.org/0009-0007-7374-9345}{https://orcid.org/0009-0007-7374-9345}}


\date{October 12, 2026}


\begin{abstract}
	We study the circle-relation subspace $O_n^\circ(V)$ in the defining
	relations of the level-$n$ Zhu algebra of a vertex operator algebra $V$.
	We assemble the residue relations generating \(O_n^\circ(V)\) into
	generating functions and obtain an explicit recursion that replaces every
	sufficiently deep mode by retained modes with indices at least $-2n-1$,
	modulo \(O_n^\circ(V)\).  This recursion is encoded in a left truncated field
	\(Y_L\).  Since such mode replacements are valid only in the
	leftmost position of an operator product, we prove the replacement,
	derivative, commutator, and one-step iterate identities for \(Y_L\) needed
	for finite-length calculations.

	For a finite homogeneous strong generating set $S$, we define the
	finite-length circle subspaces $O_n^\circ(N,S,V)$ and their plus
	subspaces. We formulate a finite-length conjecture under an explicit
	auxiliary-degree hypothesis and a separate termwise conjecture for the
	correction residues of one-sided iterates.
	We prove reductions to the
	\(k=0\) strip and to last vacuum mode \(-1\), recover the length-one
	relation spaces, and give explicit Virasoro computations at word length
	two for \(n=1\) and \(n=2\).  The $Y_L$ identities are the input for the companion
	paper (Part~II).
\end{abstract}
\maketitle

\section{Introduction}

\subsection*{Zhu algebras and higher levels}
For a vertex operator algebra (VOA) $V$, Zhu introduced the associative algebra $A_0(V)$ in his proof
of modular invariance of characters \cite{Z2}; see also Zhu's thesis \cite{Z1}.
Dong--Li--Mason generalized Zhu's construction to a family of higher-level algebras $A_n(V)=V/O_n(V)$
and established functorial links between admissible $V$--modules and $A_n(V)$--modules \cite{DLM}.
These higher-level Zhu algebras are particularly useful in settings where indecomposable representation
theory is essential; for example, modular invariance phenomena persist under $C_2$--cofiniteness without
assuming semisimplicity of $A_0(V)$ \cite{MiyC2}.

Twisted and orbifold variants are developed in
\cite{DLMtwisted,DLMmoonshine} and \cite{MiyTan,vanEkeren}. Explicit higher-level
computations appear in \cite{BVY-Heisenberg,BVY-Virasoro,AB,AB2}.
Further descriptions in \cite{DGK,BFOHY} do not give explicit presentations
by generators and relations. Such presentations can be used to construct
and analyze indecomposable modules; the passage from $A_n(V)$-modules
to admissible $V$-modules requires care even at level one \cite{BVY}.
Huang \cite{Huang} constructed an associative algebra $A^\infty(V)$ acting on
associated graded spaces of lower-bounded generalized $V$-modules. He also showed that
Zhu's algebra and the algebras $A_n(V)$ are isomorphic to subalgebras of
$A^\infty(V)$ \cite[Theorems~4.1--4.2]{Huang}.  For the defining ideal of
$A^\infty(V)$, Addabbo \cite{Addabbo} showed that the families of elements proposed by Huang do
not generate it, already for the rank-one Heisenberg vertex operator algebra.

We derive residue and truncation identities for working modulo the
circle-relation subspace
\(O_n^\circ(V)\), whose relations constitute a central difficulty in
explicit presentations for \(n\ge1\).  The truncated field \(Y_L\)
organizes the resulting deep-mode recursions.
\subsection*{The circle-relation subspace}
Fix $n\in\N$.
The level $n$ Zhu algebra is defined by $A_n(V)=V/O_n(V)$ \cite{Z2,DLM}.
Its defining subspace decomposes as
\[
\begin{aligned}
O_n(V) &= O^L(V)+O_n^\circ(V),\\
O^L(V) &= \{(L(-1)+L(0))v\mid v\in V\},\qquad
O_n^\circ(V)=\Span\{u\circ_n v\mid u,v\in V\}.
\end{aligned}
\]
Following \cite{AB}, we keep the translation-relation subspace
\(O^L(V)\) separate from the circle-relation subspace
\(O_n^\circ(V)\).  The purpose of this paper is to isolate the circle
relations and derive identities for calculations modulo
\(O_n^\circ(V)\).  This separation is motivated by explicit
higher-level calculations: for \(n\ge1\), the circle relations impose
nontrivial deep-mode recursions, and those recursions must be controlled
inside iterates and normally ordered expressions.
\subsection*{Truncation, leftmost replacement, and the first quadratic obstruction}
The circle relation $u\circ_n v\in O_n^\circ(V)$ is a residue identity that replaces sufficiently deep modes of $u$ by retained modes with indices at least $-2n-1$, modulo $O_n^\circ(V)$.
However, the usual identities of vertex algebra calculus (iterate/Jacobi expansions, commutator expansions, and normal ordering; cf.\,\cite[Chs.~2--3]{LL}) can generate arbitrarily deep modes unless they are organized with care.
Deep-mode reductions are therefore applied only to the leftmost field of a product (the operator acting last); we call this \emph{leftmost replacement}.

To exploit the deep-mode recursion efficiently, we introduce for each homogeneous $u\in V$ a left truncated field $Y_L(u,x)$, which records the image of $Y(u,x)$ in $V/O_n^\circ(V)$ while replacing all sufficiently deep modes of $u$ by explicit retained modes with indices at least $-2n-1$.
More precisely, for homogeneous $u\in V$ we write
\[
Y_L(u,x)=\sum_{j\ge -2n-1} u_j\,p_{j,u}(x),
\]
where the coefficient series $p_{j,u}(x)$ are explicit and encode the band/tail splitting determined by the level $n$ circle relations.
The field $Y_L$ encodes the circle relations attached to $u$ (Lemma~\ref{lem:defect-spans-Onu}) and is designed to be compatible with leftmost replacement in iterated computations.

A key theme already visible in the first nontrivial length-$2$ setting is that truncation behaves differently for \emph{singular} iterates $u_i v$ ($i\ge0$) and \emph{regular} iterates $u_{-r}v$ ($r\ge1$).
The $Y_L$--formalism behaves well on the singular side, whereas expanding $Y_L(u_{-r}v,x)$ typically produces monomials whose mode indices are less than $-2n-1$ when expanded using the iterate formula.
This motivates separating the relations controlled by $Y_L$ (that is, those coming from the singular iterates $u_i v$ with $i\ge0$)
from the remaining regular quadratic relations (Remark~\ref{rem:YL-safe-vs-regular}) and leads to a quadratic reduction lemma (Lemma~\ref{lem:On2-cutoff-reduction}).

\subsection*{Main results}
Throughout this subsection, \(n\in\N\) is fixed and \(V\) is a vertex
operator algebra.  For a homogeneous vector \(u\in V\), write
\(O_n^\circ(u,V)\subseteq O_n^\circ(V)\) for the span of
\[
u\,R^n_{m,k}\,a,
\qquad
m\ge k\in\N,\quad a\in V,
\]
as in Definition~\ref{def:Onu-subspace}.  In the statements below,
\(u\) and \(v\) denote homogeneous elements of \(V\).
\smallskip
\noindent\textbf{General tail recursion.}
The first main result is that every deep mode $u_{-2n-2-m}v$ ($m\ge0$)
is congruent modulo $O_n^\circ(u,V)$ to a finite linear combination
of retained modes $u_j v$ with $j\ge -2n-1$ (Proposition~\ref{prop:tail-AB1}):
\begin{equation}\label{eq:tail-intro}
u_{-2n-2-m}v
\;\equiv_{O_n^\circ(u,V)}\;
\sum_{j\ge -2n-1}
\!\!\left(
\sum_{i=0}^{2n+j+1}
\binom{\wt u+n}{i}
\binom{-\wt u-n}{2n+2+m+j-i}
\right)
u_j v.
\end{equation}
The sum is finite on each input by lower truncation; its nonzero operator coefficients need not have a uniform upper mode bound.
This holds in full generality (all weight regimes) and in the case $\wt u + n\geq 0$, specializes to the recursion of \cite{AB}.
The coefficients are explicit binomial expressions in $\wt u+n$.
These relations are precisely what $Y_L(u,x)$ encodes: the coefficient
series $p_{j,u}(x)$ in $Y_L(u,x)=\sum_{j\ge-2n-1}u_j\,p_{j,u}(x)$ are
constructed so that \eqref{eq:tail-intro} holds for every $m\ge0$.
Conversely, after reducing to $k=0$, the residue generators $u\,R^n_{m,k}\,v$ defining $O_n^\circ(u,V)$
are recovered as the coefficients of $x^{m+2n+1}$ in $(1+x)^{\wt u+n}Y(u,x)v$ for $m\ge0$
(Proposition~\ref{prop:Y-full-GF}).

\smallskip
\noindent\textbf{One-step and iterated reductions for $Y_L$.}
The left truncated field $Y_L(u,x)$ packages \eqref{eq:tail-intro} into a
field-level congruence $Y(u,x)\equiv_{O_n^\circ(u,V)} Y_L(u,x)$.
For a regular iterate $u_{-r}v$ ($r\ge1$), the one-step reduction
(Lemma~\ref{lem:one-step-general}) expresses $Y(u_{-r}v,x)$ modulo
the comparison subspace
$U:=O_n^\circ(u,V)+O_n^\circ(v,V)+\sum_{i\ge0}O_n^\circ(u_iv,V)$
entirely in terms of $Y_L$:
\[
Y(u_{-r}v,x)
\;\equiv_{U}\;
\sum_{i\ge0} F^+_{i,r,u}(x,x)\,Y_L(u_iv,x)
\;+\;
Y_L(v,x)\,Y_L(u_{-r}\mathbf{1},x),
\]
where $F^+_{i,r,u}(x,x)$ is an explicit diagonal kernel.
An iterated version (Proposition~\ref{prop:one-sided-iterate}) extends this
to vacuum words of arbitrary length, expressing $Y(w,x)$ as a sum of
products of truncated fields $Y_L$ with explicit correction kernels.

\smallskip
\noindent\textbf{Bounded-length generation and conjectures.}
Given a finite homogeneous strong generating set $S\subset V$, let
$O_n^\circ(N,S,V)$ denote the subspace of $O_n^\circ(V)$ generated by
circle-type relations attached to vacuum words of length at most $N$
(Definition~\ref{def:OnN-SV}).
We prove that generators can always be taken with $k=0$ and last vacuum
mode $-1$ (Corollary~\ref{cor:OnNSV-first-reduction}), and we establish a
quadratic reduction in the case $N=2$
(Lemma~\ref{lem:On2-cutoff-reduction}).
The proposed word-length bound is $N=n+1$, under the auxiliary-degree condition
\eqref{eq:auxiliary-degree-drop}. The precise statement is
Conjecture~\ref{conj:finite-length-zhu-relation}. The correction-residue
assertion is stated separately as Conjecture~\ref{conj:finite-level-revised}.
The correction residues and the quadratic relations are studied in the
companion paper \cite{Part2}.

\subsection*{Organization}
Section~\ref{zhu-algebra-definition-section} isolates the circle-relation
subspaces and constructs the left truncated field \(Y_L\) from the residue
identities.  Section~\ref{sec:YL-calculus} proves the leftmost replacement,
derivative, commutator, and one-step identities for \(Y_L\) needed for iterated
computations.  Section~\ref{sec:finite-length} defines replacement residues and the
finite-length circle subspaces
\(O_n^\circ(N,S,V)\), states the conditional finite-length conjecture, and proves the first reduction lemmas.  Finally,
Section~\ref{sec:examples} gives length-one and Virasoro quadratic examples
and records the outlook toward finite quadratic generating windows.

\subsection*{Notation}
The following symbols recur throughout; the right column gives the place of definition.
\begin{center}
\renewcommand{\arraystretch}{1.15}
\begin{tabular}{lll}
Symbol & Meaning & Defined in\\ \hline
$u\,R^n_{m,k}\,v$ & residue operator & Def.~\ref{def:Rmk}, \eqref{eq:Rmk-def}\\
$O_n^\circ(u,V)$ & relations from one vector $u$ & Def.~\ref{def:Onu-subspace}\\
$\equiv_W$, $\sim_n$ & congruence modulo $W$, modulo $O_n^\circ(V)$ & Def.~\ref{def:W-congruence}\\
$C_M(u)$ & coefficient of $x^M$ in $(1+x)^{\wt u+n}Y(u,x)$ & \eqref{eq:CM-def}\\
$p_{j,u}(x)$, $P_{j,u}(x)$ & $p$-series and its tail & \eqref{eq:pj-def}, \eqref{eq:p-P-decomp}\\
$p_{j,r,u}(x)$, $P_{j,r,u}(x)$ & derivatives of the $p$-series & \eqref{eq:pjmu-def}, \eqref{eq:pm-Pm-decomp}\\
$Y_L(u,x)$ & left truncated field & Def.~\ref{def:YL}\\
$F^\pm(x)$, $Y^\pm(u,x)$ & regular and singular parts & Def.~\ref{plus-minus-split}\\
$\Delta(u,x)$ & truncation difference $Y_L(u,x)-Y(u,x)$ & \eqref{eq:Delta-def}\\
$F^+_{i,u}$, $F^+_{i,r,u}$ & regular commutator kernels & \eqref{eq:Fiu-plus-def}, \eqref{eq:Fium-def}\\
$Q_{i,r,u}(x)$ & polynomial part of $F^+_{i,r,u}(x,x)$ & \eqref{eq:Qium-def}\\
$Y_L(u_{-r}\mathbf1,x)$ in products & derivative field & Conv.~\ref{conv:descendant-derivative}\\
$v\,R^{n,\widehat Y}_{m,k}\,w$ & residue of a replacement field & Def.~\ref{def:Rmk-replacement}\\
$O_n^\circ(N,S,V)$, $O_n^\circ(N,S,V)^+$ & finite-length subspaces & Def.~\ref{def:OnN-SV}\\
$\mathcal U_\ell^{(0)}$, $\mathcal U_\ell$ & moduli of the iterated expansion & Prop.~\ref{prop:one-sided-iterate}\\
\end{tabular}
\end{center}
\section{The relations in \texorpdfstring{$O^\circ_n(V)$}{On(V)} and the field \texorpdfstring{$Y_L$}{YL}}\label{zhu-algebra-definition-section}
\subsection{Definitions and the residue relations \texorpdfstring{$R^n_{m,k}$}{Rnm,k}}
In this section, letting $V$ be a vertex operator algebra, we recall the definition and some properties of the algebras $A_n(V)$ for $n \in \N$, first introduced in \cite{Z2} for $n = 0$, and then generalized to $n >0$ in \cite{DLM}.
Throughout, we take $\N = \{0,1,2,\dots\}$, and for  $\alpha\in\C$ and $i\in\N$ we use the generalized binomial
coefficients
\begin{equation}\label{eq:binom-coeff}
\binom{\alpha}{i}
=
\frac{\alpha(\alpha-1)\cdots(\alpha-i+1)}{i!},
\qquad i\ge 0,
\end{equation}
with the convention that $\binom{\alpha}{0}=1$ and $\binom{\alpha}{k}=0$ for $k<0$.

The series $(x+y)^\alpha$ is interpreted as a formal power series expanded in nonnegative
powers of $y$:
\begin{equation}\label{eq:binomial-expansion}
(x+y)^\alpha=\sum_{i\ge 0}\binom{\alpha}{i}x^{\alpha-i}y^i,
\end{equation}
where $\binom{\alpha}{i}$ is the generalized binomial coefficient \eqref{eq:binom-coeff}. We write coefficient extraction by $\bigl[x^M\bigr]F(x)=\Res_x\,x^{-M-1}F(x)$ for $F(x)\in\C((x))$ and $M\in\Z$.
\smallskip

 For $n \in \N$, let $O_n(V)$ be the subspace of $V$ spanned by elements of the form
\begin{equation}\label{define-circ}
u \circ_n v =
\res_x \frac{(1 + x)^{\mathrm{wt}\, u + n}Y(u, x)v}{x^{2n+2}}
\end{equation}
for all homogeneous $u \in V$ and for all $v \in V$, and by elements of the form $(L(-1) + L(0))v$ for all $v \in V$. The vector space $A_n(V)$ is defined to be the quotient space $V/O_n(V)$.

That is, using the notation introduced in \cite{AB},  we define
\begin{equation}\label{define-O-subsets}
O^L (V) = \{ (L(-1) + L(0))v \, | \, v \in V \} \quad
\mathrm{and}
\quad O_n^\circ(V) = \mathrm{span} \{ u \circ_n v \; | \; u,v \in V\} ,
\end{equation}
then
\begin{equation}\label{define-O}
O_n(V) =  O^L (V) + O_n^\circ(V) .
\end{equation}
We will sometimes refer to $\circ_n$ as the ``circle product''.

\begin{defn}[Congruence modulo a subspace]\label{def:W-congruence}
	Let $V$ be a vector space, and let $W\subset V$ be a subspace.  Let $V((x))$ denote the vector space of Laurent series in $x$ with coefficients in $V$.

	(1) For $a,b\in V$ we write $a\equiv_W b$ if $a-b\in W$.

	(2) For $A,B\in V((x_1))\cdots((x_r))$ we write $A\equiv_W B$ if every coefficient of
	$A-B$ lies in $W$, i.e. $A-B\in W((x_1))\cdots((x_r))$.

	(3) Let $F,G$ be linear maps from $V$ to $V((x_1))\cdots((x_r))$.
	We write $F\equiv_W G$ if
	$(F-G)v\in W((x_1))\cdots((x_r))$ for all $v\in V$; equivalently
	$F-G\in\Hom\!\bigl(V,W((x_1))\cdots((x_r))\bigr)$.
	The Laurent lower bounds may depend on the input vector.
	In particular, for ordinary operators $A,B\in\End(V)$,
	\[
	A\equiv_W B
	\quad\Longleftrightarrow\quad
	(A-B)(V)\subseteq W.
	\]
	Equivalently, $A$ and $B$ induce the same linear map from $V$ to $V/W$.
	Throughout, when $W=O_n^\circ(V)$ we also write $\sim_n$ in place of $\equiv_W$. In particular, these congruences are preserved under coefficient extraction and formal residues
	in the variables $x_1,\dots,x_r$, and under multiplication by scalars in $\C((x_1))\cdots((x_r))$.
\end{defn}

For homogeneous $u\in V$ and any $v\in V$, define
\[
u *_n v
=\sum_{m=0}^n (-1)^m \binom{m+n}{n}\,\Res_x
\frac{(1+x)^{\wt u+n}Y(u,x)v}{x^{n+m+1}}.
\]
Extend $*_n$ to all of $V$ by linearity in $u$.
It is shown in \cite{DLM} that, with this multiplication, $O_n(V)$ is a two-sided ideal of $V$ and
$A_n(V):=V/O_n(V)$
is an associative algebra, called the \emph{level $n$ Zhu algebra}.
\begin{defn}[Residue operators]\label{def:Rmk}
	Let $u\in V$ be homogeneous, $n\in\N$, and define
	\begin{equation}\label{eq:Rmk-def}
	u\,R^{n}_{m,k}\,v
	\;:=\;
	\Res_x \frac{(1+x)^{\wt u+n+k}}{x^{m+2n+2}}\,Y(u,x)v.
	\end{equation}
		In particular, the case $(m,k)=(0,0)$ recovers the original circle product:
	\[
	u\,R^{n}_{0,0}\,v \;=\; u\circ_n v
	\qquad\text{by }\eqref{eq:Rmk-def}\text{ and }\eqref{define-circ}.\]
\end{defn}
\smallskip
\noindent
The residue operators $R^n_{m,k}$ provide a uniform way to package circle-type relations:
for fixed $u$ and $v$, the elements $uR^n_{m,k}v$ lie in $O_n^\circ(V)$ for $m\ge k\ge0$ (cf.\ \cite{AB}),
and we later apply the same residues to replacement fields $Y(w,x)\equiv_U \widehat Y(w,x)$
(Definition~\ref{def:Rmk-replacement}) to produce further relations from longer words.

The following lemma is a slight refinement of a result in \cite{AB}.
\begin{lem}[Residues $uR^n_{m,k}v$ lie in $O_n^\circ(V)$]\label{Zhu-lemma}
	Let $u\in V$ be homogeneous and let $m,k\in\N$ with $m\ge k\ge 0$.
	Then for all $v\in V$, the element $u\,R^{n}_{m,k}\,v$ defined in
	\eqref{eq:Rmk-def} satisfies
	\begin{equation}\label{eq:Rmk-in-Oncirc}
	u\,R^{n}_{m,k}\,v \in O_n^\circ(V).
	\end{equation}
	Moreover, for each fixed homogeneous $u\in V$, the subspace of $O_n^\circ(V)$
	spanned by the elements $u\,R^{n}_{m,k}\,v$ (with $v\in V$ and $m\ge k\ge 0$)
	is generated by the elements with $k=0$, i.e. by
	\[
	u\,R^{n}_{m',0}\,v,
	\qquad v\in V,\; m'\in\N.
	\]
	In fact, for all $v\in V$ and $m\ge k\ge 0$ one has
	\[
	u\,R^{n}_{m,k}\,v
	=
	\sum_{t=0}^{k}\binom{k}{t}\,u\,R^{n}_{m-t,0}\,v.
	\]
\end{lem}

\begin{proof}
	The inclusion \eqref{eq:Rmk-in-Oncirc} was proved in \cite{AB}.
	It remains to show that, for fixed homogeneous $u\in V$, the subspace spanned
	by the elements $u\,R^{n}_{m,k}\,v$ is generated by those with $k=0$.

	\smallskip
	Fix $u\in V$ homogeneous and $m\ge k\ge 0$. Using the binomial expansion in \eqref{eq:binomial-expansion} and substituting it into the defining residue for $u\,R^{n}_{m,k}\,v$ gives
	\begin{align}
	u\,R^{n}_{m,k}\,v
	&=
	\Res_x \frac{(1+x)^{\wt u+n+k}}{x^{m+2n+2}}\,Y(u,x)v \notag\\
	&=
	\sum_{t=0}^{k}\binom{k}{t}\,
	\Res_x \frac{x^t(1+x)^{\wt u+n}}{x^{m+2n+2}}\,Y(u,x)v \notag\\
	&=
	\sum_{t=0}^{k}\binom{k}{t}\,
	\Res_x \frac{(1+x)^{\wt u+n}}{x^{(m-t)+2n+2}}\,Y(u,x)v. \notag
	\end{align}
	For each $t$ with $0\le t\le k$, set $m' = m-t$. Since $m\ge k$, we have
	$m'\ge m-k\ge 0$, hence $m'\in\N$. Therefore each summand is exactly
	\[
	u\,R^{n}_{m',0}\,v
	=
	\Res_x \frac{(1+x)^{\wt u+n}}{x^{m'+2n+2}}\,Y(u,x)v.
	\]
	Thus every $u\,R^{n}_{m,k}\,v$ (with $m\ge k$) is a $\C$-linear combination of
	elements $u\,R^{n}_{m',0}\,v$, proving the second claim.
\end{proof}

\subsection{The subspace \texorpdfstring{$O_n^{\circ}(u,V)\subseteq O_n^{\circ}(V)$}{On(u,V) in On(V)}}

\begin{defn}\label{def:Onu-subspace}
	Fix a homogeneous element $u\in V$ and let $n\in\N$.
	We define $O_n^\circ(u,V)$ to be the subspace of $O_n^\circ(V)$ spanned by the elements
	\[
	u\,R^{n}_{m,k}\,v \qquad (v\in V,\; m,k\in\N,\; m\ge k),
	\]
	where $u\,R^{n}_{m,k}\,v$ is as defined in \eqref{eq:Rmk-def}.
\end{defn}

\begin{lem}[Circle products of vacuum descendants span $O_n^\circ(u,V)$]\label{lem:circle-from-modes}
	Fix a homogeneous element $u\in V$ and let $n\in\N$.
	Then the subspace $O_n^\circ(u,V)$ of $O_n^\circ(V)$ defined in Definition~\ref{def:Onu-subspace}
	is generated precisely by the circle products
	\[
	(u_{-r}\mathbf{1})\circ_n v,
	\qquad v\in V,\; r\ge 1.
	\]
\end{lem}

\begin{proof}
By Lemma~\ref{Zhu-lemma}, $O_n^\circ(u,V)$ is spanned by
\begin{equation}\label{eq:k0-residues-u}
	u\,R^{n}_{r,0}\,v
	\qquad (v\in V,\; r\in\N).
	\end{equation}
Fix $m\ge1$ and $v\in V$; note $\wt(u_{-m}\mathbf1)=\wt u+m-1$. The creation and derivative properties give
$Y(u_{-m}\mathbf1,x)=\frac{1}{(m-1)!}\partial_x^{m-1}Y(u,x)$.
With $F(x)=(1+x)^{\wt u+n+m-1}x^{-2n-2}$, formal integration by parts gives
\begin{equation}\label{eq:circle-Fm-1}
	(u_{-m}\mathbf{1})\circ_n v
	=
	\frac{(-1)^{m-1}}{(m-1)!}\res_x F^{(m-1)}(x)\,Y(u,x)v.
	\end{equation}
The Leibniz rule, together with
$\partial_x^k x^{-2n-2}=(-1)^k k!\binom{2n+k+1}{k}x^{-2n-2-k}$, yields
\begin{equation}\label{eq:F-derivs-closed}
	F^{(m-1)}(x)
	=
	(m-1)!\sum_{k=0}^{m-1}
	(-1)^k
	\binom{\wt u+n+m-1}{m-1-k}
	\binom{2n+k+1}{k}\,
	(1+x)^{\wt u+n+k}x^{-2n-2-k}.
	\end{equation}
Substituting into \eqref{eq:circle-Fm-1} and using \eqref{eq:Rmk-def} gives
\begin{equation}\label{eq:circle-as-residues}
	(u_{-m}\mathbf{1})\circ_n v
	=
	\sum_{k=0}^{m-1}
	(-1)^{m-1+k}
	\binom{\wt u+n+m-1}{m-1-k}
	\binom{2n+k+1}{k}\,
	\bigl(u\,R^{n}_{k,k}\,v\bigr).
	\end{equation}
Lemma~\ref{Zhu-lemma} now puts every circle product in the span of
\eqref{eq:k0-residues-u}. Conversely,
\[
uR^n_{k,k}v=uR^n_{k,0}v+\sum_{t=1}^k\binom{k}{t}uR^n_{k-t,0}v.
\]
Thus \eqref{eq:circle-as-residues} is triangular in the $k=0$ generators,
with nonzero coefficient $\binom{2n+m}{m-1}$ at $uR^n_{m-1,0}v$.
The base case is $(u_{-1}\mathbf1)\circ_n v=uR^n_{0,0}v$; induction on $m$
gives the reverse inclusion.
\end{proof}

\begin{lem}[Change of generators for $R^{n}_{m,k}$ from $(u_{-r}\mathbf1)$ to $u$]
	\label{lem:ur1-Rmk-change-of-basis}
	Let $u\in V$ be homogeneous and let $v\in V$. Fix $n\in\N$.
	Let $m,k,r\in\N$ with $r\ge 1$.
	Then
	\begin{equation}\label{eq:ur1-Rmk-in-u-Rmk}
	\begin{aligned}
	(u_{-r}\mathbf1)\,R^{n}_{m,k}v
	&=
	\sum_{m'=m}^{m+r-1}
	\Biggl(
	\sum_{a=0}^{r-1-(m'-m)}
	(-1)^a\,
	\binom{\wt u+n+k+r-1}{a}
	\\
	&\qquad\qquad\cdot
	\binom{m+2n+r-a}{r-1-a}\,
	\binom{r-1-a}{m'-m}
	\Biggr)
	\bigl(u\,R^{n}_{m',k}v\bigr).
	\end{aligned}
	\end{equation}
\end{lem}

\begin{proof}
	Recall the defining residue (with general $k$):
	\begin{equation}\label{eq:Rmk-def-for-change-of-basis}
	u\,R^{n}_{m,k}v
	=
	\Res_x \frac{(1+x)^{\wt u+n+k}}{x^{m+2n+2}}\,Y(u,x)v.
	\end{equation}

	Since $r\ge 1$, the creation property gives
	\begin{equation}\notag
	u_{-r}\mathbf1=\frac{1}{(r-1)!}\,L(-1)^{r-1}u,
	\end{equation}
	so using $Y(L(-1)w,x)=\frac{d}{dx}Y(w,x)$ repeatedly,
	\begin{equation}\notag
	Y(u_{-r}\mathbf1,x)=\frac{1}{(r-1)!}\,\frac{d^{\,r-1}}{dx^{r-1}}Y(u,x).
	\end{equation}
	Also $\wt(u_{-r}\mathbf1)=\wt u+r-1$. Thus, by \eqref{eq:Rmk-def-for-change-of-basis},
	\begin{align}
	(u_{-r}\mathbf1)\,R^{n}_{m,k}v
	&=
	\Res_x \frac{(1+x)^{\wt(u_{-r}\mathbf1)+n+k}}{x^{m+2n+2}}\,Y(u_{-r}\mathbf1,x)v
	\notag\\
	&=
	\Res_x \frac{(1+x)^{\wt u+n+k+r-1}}{x^{m+2n+2}}\,Y(u_{-r}\mathbf1,x)v
	\notag\\
	&=
	\frac{1}{(r-1)!}\Res_x \frac{(1+x)^{\wt u+n+k+r-1}}{x^{m+2n+2}}\,
	\frac{d^{\,r-1}}{dx^{r-1}}Y(u,x)v.
	\notag
	\end{align}

	We now move the derivative off $Y(u,x)v$ by iterated residue integration-by-parts:
	iterating $\Res_x \frac{d}{dx}(\cdots)=0$ yields
	\begin{equation}\label{eq:ur1Rmk-ibp}
	(u_{-r}\mathbf1)\,R^{n}_{m,k}v
	=
	\frac{(-1)^{r-1}}{(r-1)!}\Res_x
	\frac{d^{\,r-1}}{dx^{r-1}}
	\Bigl(\frac{(1+x)^{\wt u+n+k+r-1}}{x^{m+2n+2}}\Bigr)\,
	Y(u,x)v.
	\end{equation}

	Compute the $(r-1)$-st derivative using Leibniz:
	\begin{align*}
	\frac{d^{\,r-1}}{dx^{r-1}}
	\Bigl((1+x)^{\wt u+n+k+r-1}\,x^{-(m+2n+2)}\Bigr)
	&=
	\sum_{a=0}^{r-1}\binom{r-1}{a}\,
	\frac{d^a}{dx^a}(1+x)^{\wt u+n+k+r-1}
	\notag\\
	&\qquad\cdot
	\frac{d^{\,r-1-a}}{dx^{r-1-a}}x^{-(m+2n+2)}.
	\end{align*}

	For $a\ge 0$,
	\[
	\frac{d^a}{dx^a}(1+x)^{\wt u+n+k+r-1}
	=
	a!\binom{\wt u+n+k+r-1}{a}\,(1+x)^{\wt u+n+k+r-1-a},
	\]
	and for $s\ge 0$,
	\[
	\frac{d^{\,s}}{dx^{s}}x^{-(m+2n+2)}
	=
	s!\binom{-(m+2n+2)}{s}\,x^{-(m+2n+2)-s}.
	\]
	Substitute these into the Leibniz expansion above with $s=r-1-a$.
	Using $\binom{r-1}{a}a!(r-1-a)!=(r-1)!$, the factorial cancels against
	\eqref{eq:ur1Rmk-ibp}, giving
	\begin{align}
	(u_{-r}\mathbf1)\,R^{n}_{m,k}v
	&=
	\sum_{a=0}^{r-1}
	(-1)^{r-1}\,
	\binom{\wt u+n+k+r-1}{a}\,
	\binom{-(m+2n+2)}{r-1-a}
	\notag\\
	&\qquad\cdot
	\Res_x (1+x)^{\wt u+n+k+r-1-a}\,x^{-(m+2n+r+1-a)}\,Y(u,x)v.
	\label{eq:ur1Rmk-after-deriv}
	\end{align}

	Convert the negative binomial coefficient:
	\[
	\binom{-(m+2n+2)}{r-1-a}
	=
	(-1)^{r-1-a}\binom{m+2n+r-a}{r-1-a}.
	\]
	The sign satisfies $(-1)^{r-1}(-1)^{r-1-a}=(-1)^a$, so \eqref{eq:ur1Rmk-after-deriv} becomes
	\begin{align}
	(u_{-r}\mathbf1)\,R^{n}_{m,k}v
	&=
	\sum_{a=0}^{r-1}
	(-1)^a\,
	\binom{\wt u+n+k+r-1}{a}\,
	\binom{m+2n+r-a}{r-1-a}
	\notag\\
	&\qquad\cdot
	\Res_x (1+x)^{\wt u+n+k+r-1-a}\,x^{-(m+2n+r+1-a)}\,Y(u,x)v.
	\label{eq:ur1Rmk-after-neg-bin}
	\end{align}

	Factor
	\[
	(1+x)^{\wt u+n+k+r-1-a}=(1+x)^{\wt u+n+k}\,(1+x)^{r-1-a},
	\]
	and expand
	\[
	(1+x)^{r-1-a}=\sum_{t=0}^{r-1-a}\binom{r-1-a}{t}x^t.
	\]
	Substituting into \eqref{eq:ur1Rmk-after-neg-bin} gives
	\begin{align}
	(u_{-r}\mathbf1)\,R^{n}_{m,k}v
	&=
	\sum_{a=0}^{r-1}
	(-1)^a\,
	\binom{\wt u+n+k+r-1}{a}\,
	\binom{m+2n+r-a}{r-1-a}
	\notag\\
	&\qquad\cdot
	\sum_{t=0}^{r-1-a}\binom{r-1-a}{t}\,
	\Res_x
	\frac{(1+x)^{\wt u+n+k}}{x^{m+2n+2+(r-1-a)-t}}\,Y(u,x)v.
	\label{eq:ur1Rmk-expanded-tail}
	\end{align}

	By \eqref{eq:Rmk-def-for-change-of-basis}, the residue in \eqref{eq:ur1Rmk-expanded-tail} equals
	$u\,R^{n}_{m+(r-1-a)-t,k}v$.
	Set
	\[
	m':=m+(r-1-a)-t.
	\]
	Then $m\le m'\le m+r-1$, and $t=r-1-a-(m'-m)$.
	For fixed $m'$, the condition $0\le t\le r-1-a$ becomes
	\[
	0\le r-1-a-(m'-m)\le r-1-a
	\qquad\Longleftrightarrow\qquad
	0\le a\le r-1-(m'-m).
	\]
	Also
	\[
	\binom{r-1-a}{t}=\binom{r-1-a}{r-1-a-(m'-m)}=\binom{r-1-a}{m'-m}.
	\]
	Collecting the coefficient of $u\,R^{n}_{m',k}v$ yields exactly \eqref{eq:ur1-Rmk-in-u-Rmk}.
\end{proof}

\begin{cor}[Change of generators for $\circ_n$]\label{lem:circle-Rk0-change-of-basis}
	Let $u\in V$ be homogeneous, let $n\in\N$, and let $v\in V$.
	Then for every $\ell\in\N$ with $\ell\ge 1$, we have
	\begin{equation}\label{eq:circle-in-Rk0}
	(u_{-\ell}\mathbf{1})\circ_n v
	=
	\sum_{m'=0}^{\ell-1}
	(-1)^{\ell-1+m'}\,
	\binom{2n+m'+1}{m'}\,
	\binom{\wt u+\ell-n-m'-3}{\ell-1-m'}\,
	\bigl(u\,R^{n}_{m',0}\,v\bigr).
	\end{equation}
	Moreover, for every $m'\in\N$ the inverse expansion is given by
	\begin{equation}\label{eq:Rk0-in-circle}
	u\,R^{n}_{m',0}\,v
	=
	\frac{1}{\binom{2n+m'+1}{m'}}
	\sum_{\ell=1}^{m'+1}
	\binom{\wt u-n-1}{m'+1-\ell}\,
	\bigl((u_{-\ell}\mathbf{1})\circ_n v\bigr).
	\end{equation}
\end{cor}
\begin{proof}
	\textit{Proof of \eqref{eq:circle-in-Rk0}.}
	Since $a\circ_n v=a\,R^n_{0,0}v$, apply Lemma~\ref{lem:ur1-Rmk-change-of-basis}
	with $(m,r,k)=(0,\ell,0)$. The coefficient of $u\,R^n_{m',0}v$ is
	\[
	C_{m'}
	=
	\sum_{a=0}^{\ell-1-m'}
	(-1)^a\,
	\binom{\wt u+n+\ell-1}{a}\,
	\binom{2n+\ell-a}{\ell-1-a}\,
	\binom{\ell-1-a}{m'}.
	\]
	Substitute $j:=\ell-1-a-m'$ and use the product identity
    $\binom{2n+m'+j+1}{m'+j}\allowbreak\binom{m'+j}{m'}
    =\allowbreak\binom{2n+m'+1}{m'}\allowbreak\binom{2n+m'+j+1}{j}$
	to factor $\binom{2n+m'+1}{m'}$ outside the sum, leaving
	\[
	(-1)^{\ell-1-m'}\,\binom{2n+m'+1}{m'}
	\sum_{j=0}^{\ell-1-m'}
	(-1)^{j}\,
	\binom{\wt u+n+\ell-1}{\ell-1-m'-j}\,
	\binom{2n+m'+j+1}{j}.
	\]
	Apply $\binom{2n+m'+j+1}{j}=(-1)^j\binom{-(2n+m'+2)}{j}$ and
	Vandermonde's convolution $\sum_{j=0}^{N}\binom{\alpha}{N-j}\binom{-\beta-1}{j}=\binom{\alpha-\beta-1}{N}$
	with $N=\ell-1-m'$, $\alpha=\wt u+n+\ell-1$, $\beta=2n+m'+1$.
	Since $\alpha-\beta-1=\wt u+\ell-n-m'-3$ and $(-1)^{\ell-1-m'}=(-1)^{\ell-1+m'}$,
	this yields \eqref{eq:circle-in-Rk0}.

	\smallskip
	\textit{Proof of \eqref{eq:Rk0-in-circle}.}
	Substitute \eqref{eq:circle-in-Rk0} into the right-hand side and interchange summation.
	The coefficient of $u\,R^n_{r,0}v$ is proportional to
	\[
	S
	=
	\sum_{k=0}^{m'-r}
	\binom{\wt u-n-1}{(m'-r)-k}\,
	(-1)^k\binom{\wt u+k-n-2}{k},
	\]
	where we set $k:=\ell-r-1$ and use $(-1)^k\binom{\wt u+k-n-2}{k}=\binom{n-\wt u+1}{k}$
	(the same negative-top identity \eqref{eq:binom-coeff}).
	Vandermonde's convolution gives $S=\binom{0}{m'-r}=\delta_{r,m'}$,
	so the right-hand side of \eqref{eq:Rk0-in-circle} equals $u\,R^n_{m',0}v$.
\end{proof}

\begin{lem}[$L(-1)$--descendants and $O_n^\circ(u,V)$]\label{lem:Onu-Lminus1-inclusion}
	Fix $n\in\N$ and let $u\in V$ be homogeneous. Then
	\[
	O_n^\circ(L(-1)u,V)\subseteq O_n^\circ(u,V).
	\]
\end{lem}

\begin{proof}
By Lemma~\ref{lem:circle-from-modes} and the derivative identity
$(L(-1)u)_{-r}=r\,u_{-r-1}$, for $r\ge1$ and $v\in V$,
\[
\bigl((L(-1)u)_{-r}\mathbf1\bigr)\circ_n v
=r(u_{-r-1}\mathbf1)\circ_n v\in O_n^\circ(u,V).
\]
These circle products span $O_n^\circ(L(-1)u,V)$ by the same lemma.
The useful residue identity follows from the case $r=2$, $k=0$ of
\eqref{eq:ur1-Rmk-in-u-Rmk} and the binomial reduction in
Lemma~\ref{Zhu-lemma}, since $u_{-2}\mathbf1=L(-1)u$:
\[
(L(-1)u)R^n_{m,0}v
=-(\wt u+n+1)uR^n_{m,0}v
 +(m+2n+2)uR^n_{m+1,1}v
\qquad(m\ge0).
\]
\end{proof}

\begin{rem}\label{rem:Oncirc-sum-of-Onu}
	For each homogeneous $u\in V$, let $O_n^\circ(u,V)$ be as in Definition~\ref{def:Onu-subspace}.
	Then
	\[
	O_n^\circ(V)=\sum_{u\in V\ \text{hom.}} O_n^\circ(u,V)
	\]
	as a linear span.
	Moreover, for homogeneous $u\in V$ one has
	\[
	O_n^\circ(L(-1)u,V)\subseteq O_n^\circ(u,V),
	\]
	so it is enough to let $u$ range over a subset of $V$ whose images span $V/L(-1)V$.
\end{rem}

\subsection{Mode expansions for the generators of \texorpdfstring{$O_n^\circ(u,V)$}{On(u,V)}}\label{subsec:mode-expansions}

Let $u\in V$ be homogeneous and let $n\in\N$. For $v\in V$ and $m\in\N$, expanding
$(1+x)^{\wt u+n}$ and $Y(u,x)v=\sum_{j\in\Z}u_jv\,x^{-j-1}$ gives
\begin{equation}\label{eq:mode-expansion-general}
u R^n_{m,0} v
=
\sum_{i\ge 0}
\binom{\wt u+n}{i}\,
u_{\,i-(m+2n+2)}\,v,
\end{equation}
where binomial coefficients are interpreted as in \eqref{eq:binom-coeff}.

If $\wt u+n\in\Z_{\ge0}$, then the sum is finite, and reindexing yields
\begin{equation}\notag
u R^n_{m,0} v
=
\sum_{j=-m-2n-2}^{\wt u - m - n - 2}
\binom{\wt u+n}{j+m+2n+2}\,u_j v.
\end{equation}
In particular, only modes $u_j$ with $-m-2n-2 \le j \le \wt u - m - n -2$ appear.

\subsection{The \texorpdfstring{$p$}{p}-form generating function for relations in \texorpdfstring{$O_n^\circ(u,V)$}{On(u,V)}}\label{subsec:p-form}
We next package these mode-expansion relations into a single generating function (the $p$-form), which is convenient for solving the resulting recursions globally.
\begin{prop}[Global generating--function form for $(1+x)^{\wt u+n}Y(u,x)$]\label{prop:Y-full-GF}
	Let $V$ be a vertex operator algebra, fix $n\in\N$, and let $u\in V$ be homogeneous.
	For each $M\in\Z$, set
	\begin{equation}\label{eq:CM-def}
	C_M(u):=\sum_{i\ge 0}\binom{\wt u+n}{i}\,u_{\,i-M-1}\in \End(V),
	\end{equation}
	where the sum is finite on each $v\in V$ by truncation.
	Then for all $v\in V$,
\begin{equation}\label{eq:Y-full-GF-exact}
(1+x)^{\wt u+n}Y(u,x)v=\sum_{M\in\Z} C_M(u)v\,x^M.
\end{equation}
	Moreover, in the sense of Definition~\ref{def:W-congruence},
	\begin{equation}\label{eq:Y-full-GF-cutoff}
	Y(u,x)v \equiv_{O_n^\circ(u,V)} \frac{1}{(1+x)^{\wt u+n}}\sum_{M\le 2n} C_M(u)v\,x^M.
	\end{equation}
	Equivalently, $C_M(u)v=u R^n_{m,0} v\in O_n^\circ(u,V)$ for all $v\in V$ and all $M\ge 2n+1$, where $M=m+2n+1$.
\end{prop}

\begin{proof}
Expand $Y(u,x)v$ and $(1+x)^{\wt u+n}$ using \eqref{eq:binomial-expansion}.
The coefficient at $x^M$ is
\[
[x^M]\bigl((1+x)^{\wt u+n}Y(u,x)v\bigr)
=\sum_{i\ge0}\binom{\wt u+n}{i}u_{i-M-1}v=C_M(u)v,
\]
with finitely many nonzero terms for each input by lower truncation.
This proves \eqref{eq:Y-full-GF-exact}. For $M=m+2n+1$, $m\ge0$,
\eqref{eq:Rmk-def} gives $C_M(u)v=uR^n_{m,0}v\in O_n^\circ(u,V)$.
Removing precisely these coefficients gives
\[
(1+x)^{\wt u+n}Y(u,x)v\equiv_{O_n^\circ(u,V)}\sum_{M\le2n}C_M(u)v\,x^M.
\]
Multiplication by $(1+x)^{-\wt u-n}$, as allowed by
Definition~\ref{def:W-congruence}, proves \eqref{eq:Y-full-GF-cutoff}.
\end{proof}
Throughout this work, it will be more convenient to use an alternative form, called the $p$-form, given in Proposition~\ref{prop:Y-full-p-form}.
\begin{prop}[Mode--grouped $p$--form for $Y(u,x)$]\label{prop:Y-full-p-form}
	Let $V$ be a vertex operator algebra, fix $n\in\N$, and let $u\in V$ be homogeneous.
	Let $C_M(u)$ be as in \eqref{eq:CM-def}.
	Then for every $v\in V$ we have the identity
	\begin{equation}\label{eq:CM-to-pform}
	\frac{1}{(1+x)^{\wt u + n}}\sum_{M\le 2n} C_M(u)v\,x^M
	=
	\sum_{j\ge -2n-1} u_j v\,p_{j,u}(x)
	\qquad\text{in }V((x)),
	\end{equation}
	where for each $j\ge -2n-1$ the scalar series $p_{j,u}(x)\in\C((x))$ is given by
	\begin{equation}\label{eq:pj-def}
	p_{j,u}(x)=\frac{1}{(1+x)^{\wt u + n}}\sum_{i=0}^{2n+j+1}\binom{\wt u + n}{i}\,x^{\,i-j-1}.
	\end{equation}
	Hence, in the sense of Definition~\ref{def:W-congruence},
	\begin{equation}\label{eq:Y-pform-main}
	Y(u,x)v \equiv_{O_n^\circ(u,V)} \sum_{j\ge -2n-1} u_j v\,p_{j,u}(x).
	\end{equation}
\end{prop}

\begin{proof}
	Fix $v\in V$. Starting from the left-hand side of \eqref{eq:CM-to-pform} and using \eqref{eq:CM-def} gives
	\[
	\frac{1}{(1+x)^{\wt u + n}}\sum_{M\le 2n} C_M(u)v\,x^M
	=
	\frac{1}{(1+x)^{\wt u + n}}
	\sum_{M\le 2n}\sum_{i\ge0}\binom{\wt u+n}{i}\,u_{\,i-M-1}v\,x^M.
	\]
	Reindex by $j:=i-M-1$ (equivalently $M=i-j-1$). Then $u_{i-M-1}v=u_jv$ and $x^M=x^{\,i-j-1}$, while
	$M\le 2n$ becomes $i\le 2n+j+1$. Thus the double sum equals
	\[
	\frac{1}{(1+x)^{\wt u + n}}
	\sum_{j\ge -2n-1} u_j v\left(\sum_{i=0}^{2n+j+1}\binom{\wt u+n}{i}\,x^{\,i-j-1}\right),
	\]
	which is \eqref{eq:CM-to-pform} by \eqref{eq:pj-def}.

	\smallskip
	Finally, Proposition~\ref{prop:Y-full-GF} (equation \eqref{eq:Y-full-GF-cutoff}) gives
	\[
	Y(u,x)v \equiv_{O_n^\circ(u,V)} \frac{1}{(1+x)^{\wt u + n}}\sum_{M\le 2n} C_M(u)v\,x^M,
	\]
	and combining this with \eqref{eq:CM-to-pform} yields \eqref{eq:Y-pform-main}.
\end{proof}
\subsection{Analysis of the \texorpdfstring{$p$}{p}--form: principal part and tail series}\label{subsec:pform-analysis}

Recall $p_{j,u}(x)$ from \eqref{eq:pj-def}. The next lemma isolates its principal part and
packages the remaining terms into a tail series that is convenient for differentiation.
\begin{lem}[$p$--form decomposition: principal part vs.\ tail]\label{lem:p-decomposition}
	Fix $n\in\N$ and homogeneous $u\in V$. For each $j\ge -2n-1$, there is a unique series
	$P_{j,u}(x)\in\C[[x]]$ such that for $p_{j,u}(x)$ from \eqref{eq:pj-def}, we have
\begin{equation}\label{eq:p-P-decomp}
p_{j,u}(x)=x^{-j-1}+x^{2n+1}P_{j,u}(x).
\end{equation}
	Moreover,
	\begin{equation}\label{eq:P-explicit}
	P_{j,u}(x)
	=
	-\frac{1}{(1+x)^{\wt u+n}}
	\sum_{t\ge 0}\binom{\wt u+n}{2n+j+2+t}\,x^{t}.
	\end{equation}
	Equivalently,
	\begin{align}\label{eq:P-coeff-two-forms}
	P_{j,u}(x)
	&=
	\sum_{m\ge0}
	\left(
	\sum_{i=0}^{2n+j+1}
	\binom{\wt u+n}{i}
	\binom{-\wt u-n}{2n+2+m+j-i}
	\right)x^m
	\notag\\
	&=
	-
	\sum_{m\ge0}
	\left(
	\sum_{t=0}^{m}
	\binom{-\wt u-n}{m-t}
	\binom{\wt u+n}{2n+j+2+t}
	\right)x^m.
	\end{align}
	We have $P_{j,u}(0)=-\binom{\wt u+n}{2n+j+2}$. Suppose $\wt u+n\in\Z_{\ge 0}$.
	If $\wt u+n\le 2n+j+1$, then $P_{j,u}(x)=0$, that is, $p_{j,u}(x)=x^{-j-1}$.
	If $\wt u+n\ge 2n+j+2$, then $(1+x)^{\wt u+n}P_{j,u}(x)$ is a nonzero polynomial of
	degree $\wt u-n-j-2<\wt u+n$, so $P_{j,u}(x)$ is not a polynomial.
\end{lem}

\begin{proof}
	Starting from \eqref{eq:pj-def}, multiply by $(1+x)^{\wt u+n}$ to obtain
	\[
	(1+x)^{\wt u+n}p_{j,u}(x)=\sum_{i=0}^{2n+j+1}\binom{\wt u+n}{i}\,x^{\,i-j-1}.
	\]
	On the other hand,
	\[
	(1+x)^{\wt u+n}x^{-j-1}=\sum_{i\ge 0}\binom{\wt u+n}{i}\,x^{\,i-j-1}.
	\]
	Subtracting gives
	\[
	(1+x)^{\wt u+n}\bigl(x^{-j-1}-p_{j,u}(x)\bigr)
	=
	\sum_{i\ge 2n+j+2}\binom{\wt u+n}{i}\,x^{\,i-j-1}
	=
	 x^{2n+1}\sum_{t\ge 0}\binom{\wt u+n}{2n+j+2+t}\,x^{t}.
	\]
	Dividing by $(1+x)^{\wt u+n}$ yields \eqref{eq:p-P-decomp} and \eqref{eq:P-explicit}. Uniqueness of $P_{j,u}(x)$ follows from \eqref{eq:p-P-decomp}.

	\smallskip
	To obtain the second expansion in \eqref{eq:P-coeff-two-forms}, expand
	\[
	(1+x)^{-(\wt u+n)}=\sum_{k\ge0}\binom{-\wt u-n}{k}x^k
	\]
	and substitute into \eqref{eq:P-explicit}, then collect powers of $x$.

	\smallskip
	For the first expansion in \eqref{eq:P-coeff-two-forms}, take the coefficient of $x^{2n+2+m+j}$ in
	$(1+x)^{\wt u+n}(1+x)^{-(\wt u+n)}=1$ to obtain
	\[
	0=\sum_{i=0}^{2n+2+m+j}\binom{\wt u+n}{i}\binom{-\wt u-n}{2n+2+m+j-i}.
	\]
	Splitting this sum into $0\le i\le 2n+j+1$ and $2n+j+2\le i\le 2n+2+m+j$ and reindexing the second part yields the first expansion.

	\smallskip
	Finally, let $\wt u+n\in\Z_{\ge0}$.  Then $\binom{\wt u+n}{2n+j+2+t}=0$ unless
	$2n+j+2+t\le\wt u+n$, so the sum in \eqref{eq:P-explicit} is the polynomial
	$\sum_{t=0}^{\wt u-n-j-2}\binom{\wt u+n}{2n+j+2+t}x^t$.  It is zero when
	$\wt u+n\le 2n+j+1$.  Otherwise its top coefficient is
	$\binom{\wt u+n}{\wt u+n}=1$, in degree $\wt u-n-j-2$, and
	$\wt u-n-j-2<\wt u+n$ because $j\ge-2n-1$.  If $P_{j,u}(x)$ were a polynomial, then
	$(1+x)^{\wt u+n}P_{j,u}(x)$ would be zero or have degree at least $\wt u+n$.
\end{proof}
\begin{lem}[Higher derivative structure of the $p$--series]\label{lem:p-higher-derivatives}
	Fix $n\in\N$ and homogeneous $u\in V$. Let $j\ge -2n-1$ and $r\in\N$ with $r\ge 1$.
	Let $p_{j,u}(x)$ be the series from \eqref{eq:pj-def}. Define
	\begin{equation}\label{eq:pjmu-def}
	p_{j,r,u}(x):=\frac{1}{(r-1)!}\frac{d^{r-1}}{dx^{r-1}}\,p_{j,u}(x)\in\C((x)).
	\end{equation}
	\begin{enumerate}
		\item[(i)] More explicitly, using \eqref{eq:pjmu-def} and \eqref{eq:pj-def}, we have the ``pulled-out factor'' form
	\begin{equation}\label{eq:pm-factor-form}
	\begin{aligned}
	p_{j,r,u}(x)
	&=
	\frac{1}{(1+x)^{\wt u+n+r-1}}
	\sum_{s=0}^{r-1}
	\binom{-\wt u-n}{r-1-s}(1+x)^s\\
	&\qquad\cdot
	\sum_{i=0}^{2n+j+1}
	\binom{\wt u+n}{i}\binom{i-j-1}{s}\,x^{\,i-j-1-s}.
	\end{aligned}
	\end{equation}

				\item[(ii)] There is a unique series $P_{j,r,u}(x)\in\C[[x]]$ such that
		\begin{equation}\label{eq:pm-Pm-decomp}
		p_{j,r,u}(x)=\binom{-j-1}{r-1}x^{-j-r}+x^{2n+2-r}P_{j,r,u}(x).
		\end{equation}

		Moreover,
		\begin{equation}\label{eq:Pm-coeff-scaling}
		P_{j,r,u}(x)
		=
		\sum_{t\ge0}\binom{2n+1+t}{r-1}
		\left(
		\sum_{i=0}^{2n+j+1}
		\binom{\wt u+n}{i}\binom{-\wt u-n}{2n+2+t+j-i}
		\right)x^t.
		\end{equation}
		In particular, if $0\le\wt u+n\le 2n+j+1$, then $P_{j,r,u}(x)=0$.
	\end{enumerate}
\end{lem}
\begin{proof}
	(i) Apply Leibniz' rule directly to the product in \eqref{eq:pj-def}, divide by $(r-1)!$ as in
	\eqref{eq:pjmu-def}, and use
	\[
	\frac{1}{q!}\frac{d^q}{dx^q}(1+x)^{-(\wt u+n)}=\binom{-\wt u-n}{q}(1+x)^{-(\wt u+n)-q},\,
	\frac{1}{s!}\frac{d^s}{dx^s}x^{\,i-j-1}=\binom{i-j-1}{s}x^{\,i-j-1-s}.
	\]
	Factoring $(1+x)^{-(\wt u+n+r-1)}$ yields \eqref{eq:pm-factor-form}.

	\smallskip
	(ii) By Lemma~\ref{lem:p-decomposition} (equation \eqref{eq:p-P-decomp}) we have
	$p_{j,u}(x)=x^{-j-1}+x^{2n+1}P_{j,u}(x)$ with $P_{j,u}(x)\in\C[[x]]$.
	Applying $\frac{1}{(r-1)!}\frac{d^{r-1}}{dx^{r-1}}$ and using
	$\frac{1}{(r-1)!}\frac{d^{r-1}}{dx^{r-1}}x^\alpha=\binom{\alpha}{r-1}x^{\alpha-(r-1)}$
	gives the principal part $\binom{-j-1}{r-1}x^{-j-r}$; factoring $x^{2n+2-r}$ from the differentiated
	$x^{2n+1}P_{j,u}(x)$ defines $P_{j,r,u}(x)$ and yields \eqref{eq:pm-Pm-decomp}, with uniqueness clear.
	Thus
	\[
	x^{2n+2-r}P_{j,r,u}(x)=\frac{1}{(r-1)!}\frac{d^{r-1}}{dx^{r-1}}\bigl(x^{2n+1}P_{j,u}(x)\bigr).
	\]

	\smallskip
	For \eqref{eq:Pm-coeff-scaling}, substitute the first expansion in \eqref{eq:P-coeff-two-forms} for $P_{j,u}(x)$,
	multiply by $x^{2n+1}$, differentiate term-by-term, divide by $(r-1)!$, and then factor $x^{2n+2-r}$.
	\smallskip
	Finally, if $0\le\wt u+n\le 2n+j+1$, then $P_{j,u}(x)=0$ by Lemma~\ref{lem:p-decomposition},
	and the last display gives $P_{j,r,u}(x)=0$.
\end{proof}
The preceding results show that, for each homogeneous $u$, the residue
relations defining the subspace $O_n^\circ(u,V)\subset O_n^\circ(V)$
determine all sufficiently deep modes $u_{-2n-2-m}$ in terms of the
retained modes $u_j$ with $j\ge -2n-1$, with only finitely many acting nontrivially on each input. It is therefore natural to
replace $Y(u,x)$ by a canonical truncated field which records exactly
this data in the quotient $V/O_n^\circ(u,V)((x))$.
\subsection{Left truncated fields \texorpdfstring{$Y_L(u,x)$}{YL(u,x)}: tail recursion and weight regimes}

It will be convenient to package the image of $Y(u,x)$ modulo $O_n^\circ(u,V)$ as a truncated
field depending on $n$. This field encodes precisely the relations arising from the
subspaces $O_n^\circ(u,V)$ introduced in Definition~\ref{def:Onu-subspace} and the residue
relations of Lemma~\ref{Zhu-lemma}.

\begin{defn}[Left truncated field]\label{def:YL}
	Fix $n\in\N$ and let $u\in V$ be homogeneous. Define the \emph{left truncated field}
	\begin{equation}\label{eq:YL-pform-def}
	Y_L(u,x):=\sum_{j\ge -2n-1} u_j\,p_{j,u}(x)\ \in\ \Hom\!\bigl(V,V((x))\bigr),
	\end{equation}
	where $p_{j,u}(x)\in\C((x))$ is given by \eqref{eq:pj-def} (and $p_{j,u}(x)=0$ for $j<-2n-1$).
	For an arbitrary $u\in V$, $Y_L(u,x)$ is defined by linearity from the homogeneous
	components of $u$; this is well defined because $p_{j,u}(x)$ depends on $u$ only through $\wt u$.
\end{defn}

\begin{lem}[Comparing $Y$ and $Y_L$ modulo $O_n^\circ(V)$]\label{lem:Y-tilde-YL}
	Let $n\in\N$ and let $u\in V$ be homogeneous. Then, in the sense of Definition~\ref{def:W-congruence},
	\[
	Y(u,x)\equiv_{O_n^\circ(u,V)} Y_L(u,x).
	\]
	In particular, since $O_n^\circ(u,V)\subseteq O_n^\circ(V)$, we have $Y(u,x)\sim_n Y_L(u,x)$.
\end{lem}
\begin{proof}
Equation~\eqref{eq:Y-pform-main} gives
$Y(u,x)\equiv_{O_n^\circ(u,V)}\sum_{j\ge-2n-1}u_jp_{j,u}(x)=Y_L(u,x)$
by Definition~\ref{def:YL}. The final claim uses
$O_n^\circ(u,V)\subseteq O_n^\circ(V)$ from Definition~\ref{def:Onu-subspace}.
\end{proof}

\begin{defn}\label{plus-minus-split}
	Let $F(x)=\sum_{t\in\Z}F_t x^t\in \End(V)[[x,x^{-1}]]$. Define
	\[
	F^+(x):=\sum_{t\ge0}F_t x^t,\qquad F^-(x):=\sum_{t<0}F_t x^t.
	\]
	In particular, for $Y(u,x)=\sum_{k\in\Z}u_k x^{-k-1}$ we set
	\[
	Y^+(u,x):=\sum_{k\le -1}u_k x^{-k-1},\qquad
	Y^-(u,x):=\sum_{k\ge 0}u_k x^{-k-1}.
	\]
\end{defn}
\subsection{Solving the recursion in \texorpdfstring{$O_n^\circ(u,V)$}{On(u,V)}}
Now we solve the recursion presented in \eqref{eq:mode-expansion-general} using Proposition \ref{prop:Y-full-p-form}.

\begin{prop}[Tail recursion for $Y(u,x)$]\label{prop:tail-AB1}
	Let $V$ be a vertex operator algebra, fix $n\in\N$, and let $u\in V$ be homogeneous.
	Using the decomposition \eqref{eq:p-P-decomp} from Lemma~\ref{lem:p-decomposition}, the following hold.

	\smallskip
	\noindent{\rm(1) General tail recursion (all weight regimes).}
	For every $v\in V$ and $m\in\N$,
	\begin{equation}\label{eq:tail-general}
	u_{-2n-2-m}v
	\;\equiv_{O_n^\circ(u,V)}\;
	\sum_{j\ge -2n-1}
	\left(
	\sum_{i=0}^{\,2n+j+1}
	\binom{\wt u+n}{i}\,
	\binom{-\wt u-n}{\,2n+2+m+j-i\,}
	\right)
	u_j v,
	\end{equation}
	where for each fixed $v\in V$ the sum over $j$ is finite by truncation.

	\smallskip
	\noindent{\rm(2) Field-level form.}
	Writing $Y(u,x)v=\sum_{r\in\Z}u_r v\,x^{-r-1}$ and separating the deep tail
	$r=-2n-2-m$ ($m\ge0$), we obtain
\begin{equation}\label{eq:Y-from-tail-general}
\begin{aligned}
Y(u,x)v
&=
\sum_{r\ge -2n-1} u_r v\,x^{-r-1}
+\sum_{m\ge0} u_{-2n-2-m}v\,x^{2n+1+m}\\
&\equiv_{O_n^\circ(u,V)}
\sum_{r\ge -2n-1} u_r v\,x^{-r-1}
+\sum_{m\ge0} x^{2n+1+m}
\sum_{j\ge -2n-1} \\
 &\quad\,\left(
\sum_{i=0}^{2n+j+1}
\binom{\wt u+n}{i}
\binom{-\wt u-n}{2n+2+m+j-i}
\right)u_j v.
\end{aligned}
\end{equation}
\end{prop}
\begin{proof}
	Fix $v\in V$. By Lemma~\ref{lem:Y-tilde-YL} and Definition~\ref{def:YL} we have, in the sense of
	Definition~\ref{def:W-congruence},
	\[
	Y(u,x)v \equiv_{O_n^\circ(u,V)} Y_L(u,x)v
	=
	\sum_{j\ge -2n-1} u_j v\,p_{j,u}(x).
	\]
	By Lemma~\ref{lem:p-decomposition}, for each $j\ge-2n-1$,
	\[
	p_{j,u}(x)=x^{-j-1}+x^{2n+1}P_{j,u}(x),\qquad P_{j,u}(x)\in\C[[x]].
	\]
	Substituting gives
	\begin{equation}\label{eq:YL-band-tail-split}
	Y_L(u,x)v
	=
	\sum_{j\ge -2n-1} u_j v\,x^{-j-1}
	\;+\;
	x^{2n+1}\sum_{j\ge -2n-1} u_j v\,P_{j,u}(x).
	\end{equation}

For $m\ge0$, coefficient extraction gives
\[
\begin{aligned}
u_{-2n-2-m}v
&=[x^{2n+1+m}]Y(u,x)v\\
&\equiv_{O_n^\circ(u,V)}[x^{2n+1+m}]Y_L(u,x)v\\
&=\sum_{j\ge-2n-1}\bigl([x^m]P_{j,u}(x)\bigr)u_jv\\
&=\sum_{j\ge-2n-1}\left(
\sum_{i=0}^{2n+j+1}\binom{\wt u+n}{i}
\binom{-\wt u-n}{2n+2+m+j-i}\right)u_jv.
\end{aligned}
\]
The congruence follows from Lemma~\ref{lem:Y-tilde-YL}.
The third line uses \eqref{eq:YL-band-tail-split}: its first sum has
degree at most $2n$, and its second contributes the coefficient
$[x^m]P_{j,u}(x)$.  The last line is the first coefficient formula in
\eqref{eq:P-coeff-two-forms}, proving \eqref{eq:tail-general}.
Separating the terms with indices at least $-2n-1$ in $Y(u,x)v$
and substituting \eqref{eq:tail-general} in the remaining sum gives
\eqref{eq:Y-from-tail-general}.
\end{proof}
We record several decompositions of $Y_L(u,x)$ that will be used repeatedly.  The structure of $Y_L(u,x)$ has some notable differences depending on the value of $\wt u +n$.
\medskip

\begin{rem}[Casewise structure of $Y_L(u,x)$]\label{rem:YL-cases}
	Fix \(n\in\N\) and let \(u\in V\) be homogeneous.  By
	Definition~\ref{def:YL} and Lemma~\ref{lem:p-decomposition},
	\begin{equation}\label{Y_L-split-global}
	Y_L(u,x)
	=
	\sum_{j\ge -2n-1}u_jx^{-j-1}
	+
	x^{2n+1}\sum_{j\ge -2n-1}u_jP_{j,u}(x),\quad  	Y_L^-(u,x)=Y^-(u,x)
	\end{equation}
		so the singular part is unchanged and the difference between \(Y_L(u,x)\)
		and \(Y(u,x)\) occurs in the regular part.

	\noindent{\rm(1) Case \(\wt u<-n\).}
	The negative binomial series is not completed by the finite sum defining
	\(p_{j,u}(x)\).  Hence every retained mode \(j\ge -2n-1\) may contribute to
	the correction term in \eqref{Y_L-split-global}; this is the infinite-tail
	case.

	\noindent{\rm(2) Case \(\wt u=-n\).}
	Here \(p_{j,u}(x)=x^{-j-1}\) for all \(j\ge -2n-1\).  Thus \(Y_L\) is obtained
	from \(Y\) by deleting the modes with indices less than $-2n-1$.

	\noindent{\rm(3) Case \(\wt u>-n\).}
	Set
	\[
	j_0:=\wt u-n-1.
	\]
	For \(j\ge j_0\), the finite binomial sum in \(p_{j,u}(x)\) is complete, so
	\(p_{j,u}(x)=x^{-j-1}\).  Therefore
	\begin{equation}\label{eq:YL-3a-p}
	Y_L(u,x)
	=
	\sum_{j=-2n-1}^{j_0-1}u_jp_{j,u}(x)
	+
	\sum_{j\ge j_0}u_jx^{-j-1}.
	\end{equation}
	Within this case we distinguish the two subcases used below.

	\noindent{\rm(3a) Subcase \(-n<\wt u\le n+1\).}
	Then \(j_0\le0\), so all singular modes \(j\ge0\) occur with their usual
	coefficients \(x^{-j-1}\), and only finitely many regular retained modes
	require corrected coefficients.

	\noindent{\rm(3b) Subcase \(\wt u\ge n+2\).}
	Then \(j_0\ge1\), so finitely many singular retained modes
	\(0\le j\le j_0-1\) also occur with corrected coefficients.  Nevertheless
	\eqref{Y_L-split-global} still holds after collecting the full singular part:
	the correction terms lie in the regular component of \(Y_L(u,x)\).
\end{rem}

\begin{prop}[AB specialization]\label{prop:AB-specialization}
	Let $V$ be a vertex operator algebra, fix $n\in\N$, and let $u\in V$ be homogeneous.
	Suppose that $\wt u+n\in\Z_{\ge0}$, and interpret binomial coefficients as in \eqref{eq:binom-coeff}.
	Then, in this weight regime, the general tail recursion \eqref{eq:tail-general} of
	Proposition~\ref{prop:tail-AB1} coincides with the recursion obtained in \cite{AB},
	which we record (in a convenient normalization) as follows.

	\smallskip
	\noindent{\rm AB mode recursion.}
	In the regime $\wt u+n\in\Z_{\ge0}$, for all $m\ge n+1$, \eqref{eq:tail-general} agrees with the recursion found in \cite{AB}
	(for $n+1\le m\le \wt u+2n$ the mode $u_{\wt u-m-1}$ has index at least $-2n-1$ and
	\eqref{recursion-n-equation-general} reduces to $u_{\wt u-m-1}v=u_{\wt u-m-1}v$, so its content is
	the range $m\ge\wt u+2n+1$):
	\begin{equation}\label{recursion-n-equation-general}
	u_{\wt u - m - 1}v
	\ \sim_n\
	(-1)^{m+\wt u}
	\sum_{\ell=1}^{n+\wt u}
	\binom{m-n-1}{\,\ell-1\,}
	\binom{m-n-\ell-1}{\,n+\wt u - \ell\,}
	u_{\wt u - n - \ell - 1} v.
	\end{equation}
	In particular, substituting $m=\wt u+2n+k+1$ ($k\ge0$) yields the convenient deep--tail form
	\begin{equation}\label{recursion-n-equation-2}
	u_{-2n-k-2 }v \ \sim_n\
	(-1)^{k+1}
	\sum_{j = 1}^{n+ \wt u }
	\binom{\wt u + n + k}{j-1}
	\binom{\wt u + n + k-j}{n+\wt u -j }
	u_{\wt u - n - j - 1} v ,
	\; k\ge 0.
	\end{equation}

	\smallskip
	In particular, taking the regular part of $Y(u,x)v=\sum_{r\in\Z}u_r v\,x^{-r-1}$
	recovers the formula in \cite{AB} recorded in \eqref{Y+}:
	\begin{multline}\label{Y+}
	Y^+(u,x)v \sim_n
	\sum_{k=1}^{n - \wt u + 1}  u_{-k} v x^{k-1} \\
	+ \sum_{m\ge \max\{n+1,\wt u\}} \sum_{k = n - \wt u + 2}^{2n + 1} (-1)^{m- \wt u}
	\binom{m-n-1}{k-n + \wt u -2}
	\binom{m-k-\wt u}{2n-k + 1 }
	u_{-k} v x^{-\wt u  + m}.
	\end{multline}
	The first sum is empty when $n-\wt u+1<1$. For $\wt u\ge n+2$,
	the inner sum retains some singular modes, whose coefficients here have only regular powers of $x$.
\end{prop}

\begin{proof}
Fix $v\in V$ and $k\ge0$.  All congruences are modulo $O_n^\circ(V)$.
If $\wt u+n=0$, \eqref{eq:tail-general} has zero coefficients and
\eqref{recursion-n-equation-2} has an empty sum.  Suppose $\wt u+n\ge1$.
For $j\ge\wt u-n-1$, the coefficient in \eqref{eq:tail-general} is
\[
\sum_{i=0}^{\wt u+n}\binom{\wt u+n}{i}
\binom{-\wt u-n}{2n+2+k+j-i}
=\binom{0}{2n+2+k+j}=0.
\]
For the remaining indices put $j=\wt u-n-\ell-1$,
$1\le\ell\le\wt u+n$.  The coefficient calculation is
\[
\begin{aligned}
&\sum_{i=0}^{\wt u+n-\ell}\binom{\wt u+n}{i}
 \binom{-\wt u-n}{\wt u+n+k+1-\ell-i}\\
&\quad=\frac{\ell(k+1)}{(\wt u+n)(\wt u+n+k+1-\ell)}
 \binom{\wt u+n}{\wt u+n-\ell}\binom{-\wt u-n}{k+1}\\
&\quad=(-1)^{k+1}\binom{\wt u+n+k}{\ell-1}
 \binom{\wt u+n+k-\ell}{\wt u+n-\ell}.
\end{aligned}
\]
The first display is Vandermonde's identity.  The second uses the
finite convolution identity, for integers $a\ge1$, $T>p\ge0$ and $p<a$,
\[
\sum_{i=0}^{p}\binom ai\binom{-a}{T-i}
=\frac{(a-p)(T-p)}{aT}\binom ap\binom{-a}{T-p}.
\]
Indeed, the expression on the right at $p$ minus its value at $p-1$
is $\binom ap\binom{-a}{T-p}$, by cancellation of factorials, and its
value at $p=0$ is $\binom{-a}{T}$; summing these differences proves
the identity.  Apply it with $a=\wt u+n$, $p=\wt u+n-\ell$ and
$T=\wt u+n+k+1-\ell$.  The last equality uses
$\binom{-\wt u-n}{k+1}=(-1)^{k+1}\binom{\wt u+n+k}{k+1}$
and cancels factorials.  Substitution in \eqref{eq:tail-general}
proves \eqref{recursion-n-equation-2}.

Equation~\eqref{recursion-n-equation-general} is the AB recursion
quoted from \cite{AB}; its deep-mode case is
\eqref{recursion-n-equation-2} with $m=\wt u+2n+k+1$.
For the regular-field formula, start from
\[
Y^+(u,x)v=\sum_{k=1}^{n-\wt u+1}u_{-k}v\,x^{k-1}
 +\sum_{k\ge\max\{1,n-\wt u+2\}}u_{-k}v\,x^{k-1}.
\]
In the second sum use $m=\wt u+k-1$ and apply
\eqref{recursion-n-equation-general}.  Reindex its inner sum by
$\ell=k-n+\wt u-1$; then $\ell-1=k-n+\wt u-2$ and
$n+\wt u-\ell=2n-k+1$, giving exactly \eqref{Y+}.
\end{proof}
\begin{rem}[Polynomiality and residue replacement for $Y_L$]
	\label{rem:YL-residue-tautology}\label{rem:YL-tautology}
	Fix \(n\in\N\) and homogeneous \(u\in V\).  By Definition~\ref{def:YL},
	\((1+x)^{\wt u+n}Y_L(u,x)v\) is a Laurent polynomial in \(x\) with no powers above \(x^{2n}\) for every
	\(v\in V\).  Hence, for all \(m\ge k\ge0\),
	\[
	\Res_x\frac{(1+x)^{\wt u+n+k}}{x^{m+2n+2}}Y_L(u,x)v=0.
	\]
	Since \(Y(u,x)\equiv_{O_n^\circ(u,V)}Y_L(u,x)\) by
	Lemma~\ref{lem:Y-tilde-YL}, the defining residue \(uR^n_{m,k}v\) may be
	computed using \(Y_L\) modulo \(O_n^\circ(u,V)\) and therefore vanishes there.
	More generally, if \(Y(w,x)\equiv_U\widehat Y(w,x)\), then
	Definition~\ref{def:Rmk-replacement} applies the same residue operators to
	\(\widehat Y(w,x)\) to produce relations compatible with the chosen truncation
	regime.
\end{rem}

\begin{Exa}[Left truncated Virasoro field for $n=2$]\label{exa:YL-omega-n2}
	Let $V$ be a vertex operator algebra with conformal vector $\omega$.
	Recall
	\[
	Y(\omega,x)=\sum_{m\in\Z}L(m)x^{-m-2},
	\qquad \omega_j=L(j-1),\ \wt\omega=2.
	\]
	Fix $n=2$. Then $\wt\omega+n=4$ and $-n<\wt\omega\le n+1$, so we are in
	case {\rm(3a)} of Remark~\ref{rem:YL-cases}. Hence
	\[
	Y_L(\omega,x)v=\sum_{j\ge -5}\omega_jv\,p_{j,\omega}(x),
	\qquad
	p_{j,\omega}(x)=\frac{1}{(1+x)^4}\sum_{i=0}^{j+5}\binom{4}{i}x^{\,i-j-1}.
	\]

	Since $\binom{4}{i}=0$ for $i>4$, the only $p_{j,\omega}(x)$ different from $x^{-j-1}$ occur for
	$-5\le j\le -2$. It is convenient to record them, together with $j=-1$, in the equivalent form
	$(1+x)^4p_{j,\omega}(x)$:
	\[
	\renewcommand{\arraystretch}{1.15}
	\begin{array}{c|c}
	j & (1+x)^4\,p_{j,\omega}(x)\\ \hline
	-5 & x^4\\
	-4 & x^3+4x^4\\
	-3 & x^2+4x^3+6x^4\\
	-2 & x+4x^2+6x^3+4x^4\\
	-1 & 1+4x+6x^2+4x^3+x^4=(1+x)^4
	\end{array}
	\]
Here $j_0=\wt\omega-n-1=-1$. Thus by \eqref{eq:YL-3a-p} we have
$p_{j,\omega}(x)=x^{-j-1}$ for all $j\ge -1$, and in particular $p_{-1,\omega}(x)=1$.

	Substituting $\omega_j=L(j-1)$ yields, for every $v\in V$,
	\[
	\begin{aligned}
	Y_L(\omega,x)v
	&=
	\frac{1}{(1+x)^4}\Big(
	L(-6)v\,x^4
	+L(-5)v\,(x^3+4x^4)
	+L(-4)v\,(x^2+4x^3+6x^4)\\
	&\hspace{1em}
	+L(-3)v\,(x+4x^2+6x^3+4x^4)
	+L(-2)v\,(1+x)^4
	\Big)
	+\sum_{m\ge-1}L(m)v\,x^{-m-2}.
	\end{aligned}
	\]

	Thus for level $n=2$, the left truncated Virasoro field $Y_L(\omega,x)$ modifies
	only the finitely many modes $L(-6),\dots,L(-3)$; all $L(m)$ with $m\ge-2$ appear
	with their original monomial coefficients $x^{-m-2}$.
\end{Exa}

\section{Calculus for the left truncated field \texorpdfstring{$Y_L$}{YL}}\label{sec:YL-calculus}

Throughout this section we fix $n\in\N$.  For homogeneous $u\in V$, $Y_L(u,x)$ is the left
truncated field of Definition~\ref{def:YL}, extended linearly to all of $V$.  Deep-mode
replacements modulo $O_n^\circ(V)$ are valid only for the leftmost factor of a product (the
operator acting last).  This section records the replacement, derivative, commutator, and
one-step identities for $Y_L$ that respect this restriction.

\subsection{The truncation difference}

For $u\in V$ define the \emph{truncation difference}
\begin{equation}\label{eq:Delta-def}
\Delta(u,x):=Y_L(u,x)-Y(u,x).
\end{equation}

\begin{prop}[Basic properties of $Y_L$]\label{prop:YL-On-truncation}
	Let $u\in V$ be homogeneous.
	\begin{enumerate}[label=\textup{(\arabic*)}]
		\item \textup{(Band and tail)} With $P_{j,u}(x)\in\C[[x]]$ as in \eqref{eq:p-P-decomp},
		\begin{equation}\label{eq:left-level-YL}
		Y_L(u,x)
		=
		\sum_{j=-2n-1}^{-1}u_j\,p_{j,u}(x)
		+\sum_{j\ge0}u_j\bigl(x^{-j-1}+x^{2n+1}P_{j,u}(x)\bigr),
		\end{equation}
		and $p_{j,u}(x)\in\C[[x]]$ for $-2n-1\le j\le-1$.
		\item \textup{(Singular part)} $Y_L^-(u,x)=Y^-(u,x)$, and
		\begin{equation}\label{eq:YL-plus-band-tail}
		Y_L^+(u,x)
		=
		\sum_{j=-2n-1}^{-1}u_j\,p_{j,u}(x)
		+x^{2n+1}\sum_{j\ge0}u_j\,P_{j,u}(x).
		\end{equation}
		\item \textup{(Truncation difference)}
		\[
		\Delta(u,x)\in\Hom\!\bigl(V,O_n^\circ(u,V)\bigr)[[x]].
		\]
		In particular $Y(u,x)\equiv_{O_n^\circ(u,V)}Y_L(u,x)$ and $Y(u,x)\sim_n Y_L(u,x)$.
	\end{enumerate}
\end{prop}

\begin{proof}
	\textup{(1)} By Definition~\ref{def:YL}, $Y_L(u,x)=\sum_{j\ge-2n-1}u_j\,p_{j,u}(x)$.  For
	$-2n-1\le j\le-1$, every exponent $i-j-1$ in \eqref{eq:pj-def} is nonnegative and
	$(1+x)^{-\wt u-n}\in\C[[x]]$, so $p_{j,u}(x)\in\C[[x]]$.  For $j\ge0$, \eqref{eq:p-P-decomp}
	gives $p_{j,u}(x)=x^{-j-1}+x^{2n+1}P_{j,u}(x)$.

	\textup{(2)} Apply Definition~\ref{plus-minus-split} to \eqref{eq:left-level-YL}.  For each
	$w\in V$ the sum $Y_L(u,x)w$ is finite, since $u_jw=0$ for $j\gg0$.  The only negative powers
	of $x$ in \eqref{eq:left-level-YL} are the terms $u_jx^{-j-1}$ with $j\ge0$, so
	$Y_L^-(u,x)=\sum_{j\ge0}u_jx^{-j-1}=Y^-(u,x)$, and the remaining terms give
	\eqref{eq:YL-plus-band-tail}.

	\textup{(3)} By \textup{(2)}, $\Delta(u,x)=Y_L^+(u,x)-Y^+(u,x)=\sum_{t\ge0}D_tx^t$ with
	$D_t\in\End(V)$.  Lemma~\ref{lem:Y-tilde-YL} gives $\Delta(u,x)v\in O_n^\circ(u,V)((x))$ for
	every $v\in V$, that is, $D_tv\in O_n^\circ(u,V)$ for all $t\ge0$ and $v\in V$.  The last
	statement uses $O_n^\circ(u,V)\subseteq O_n^\circ(V)$.
\end{proof}

\begin{lem}[The coefficients of $\Delta(u,x)$ span $O_n^\circ(u,V)$]\label{lem:defect-spans-Onu}
	Let $u\in V$ be homogeneous.  Then
	\begin{equation}\label{eq:Delta-coeff-span}
	\Span_\C\Bigl\{\,[x^t]\bigl(\Delta(u,x)v\bigr):\ t\in\Z,\ v\in V\,\Bigr\}
	=
	O_n^\circ(u,V).
	\end{equation}
	Consequently, $O_n^\circ(u,V)$ is the smallest subspace $U\subseteq V$ with
	$Y(u,x)\equiv_U Y_L(u,x)$.
\end{lem}

\begin{proof}
By \eqref{eq:Delta-def}, Proposition~\ref{prop:Y-full-GF},
\eqref{eq:CM-to-pform}, and \eqref{eq:Rmk-def}, for every $v\in V$,
\[
\begin{aligned}
(1+x)^{\wt u+n}\Delta(u,x)v
&=-\sum_{M\ge2n+1}C_M(u)v\,x^M\\
&=-\sum_{m\ge0}(uR^n_{m,0}v)x^{m+2n+1}.
\end{aligned}
\]
Multiplication by $(1+x)^{\wt u+n}$ preserves the span of the
coefficients: each new coefficient is a finite linear combination of
the original ones, and multiplication by $(1+x)^{-\wt u-n}$ proves
the reverse inclusion.  Thus the coefficients of $\Delta(u,x)v$,
as $v$ varies, span exactly the $k=0$ residues, which span
$O_n^\circ(u,V)$ by Lemma~\ref{Zhu-lemma}.
Finally, $Y(u,x)\equiv_UY_L(u,x)$ means that all these coefficients
belong to $U$ (Definition~\ref{def:W-congruence}(3)), proving minimality.
\end{proof}

\begin{lem}[Leftmost replacement]\label{cor:left-replacement-YL}
	Let $u_1\in V$ be homogeneous and let $u_2,\dots,u_r,v\in V$.  Then
	\begin{equation}\label{eq:leftmost-YL}
	Y(u_1,x_1)\,Y(u_2,x_2)\cdots Y(u_r,x_r)v
	\;\equiv_{\,O_n^\circ(u_1,V)\,}\;
	Y_L(u_1,x_1)\,Y(u_2,x_2)\cdots Y(u_r,x_r)v
	\end{equation}
	in $V((x_1))\cdots((x_r))$.  In particular, the two sides are congruent modulo
	$O_n^\circ(V)$.
\end{lem}

\begin{proof}
	Set $w:=Y(u_2,x_2)\cdots Y(u_r,x_r)v\in V((x_2))\cdots((x_r))$.  By \eqref{eq:Delta-def}, the
	difference of the two sides of \eqref{eq:leftmost-YL} is $\Delta(u_1,x_1)w$.  Each coefficient
	$c\in V$ of $w$ contributes $\Delta(u_1,x_1)c$, whose coefficients lie in $O_n^\circ(u_1,V)$
	by Proposition~\ref{prop:YL-On-truncation}\textup{(3)}.
\end{proof}

\begin{rem}[Leftmost versus inner replacement]\label{rem:leftmost-not-all}
	Replacing an inner factor $Y(u_k,x_k)$, $k\ge2$, produces
	$Y(u_1,x_1)\cdots\Delta(u_k,x_k)\cdots Y(u_r,x_r)v$.  The coefficients of
	$\Delta(u_k,x_k)\cdots Y(u_r,x_r)v$ lie in $O_n^\circ(u_k,V)$, but the modes of
	$u_1,\dots,u_{k-1}$ need not preserve $O_n^\circ(V)$, so there is no reason for this term to
	lie in $O_n^\circ(V)((x_1))\cdots((x_r))$.  We never replace a non-leftmost factor.
\end{rem}

\begin{lem}[$L(-1)$--descendants]\label{lem:YL-descendants}
	Let $u\in V$ be homogeneous and let $m\ge1$.
	\begin{enumerate}[label=\textup{(\alph*)}]
		\item For the untruncated field of the descendant,
		\begin{equation}\label{eq:Y-descendants-via-YL-derivative}
		Y(u_{-m}\mathbf{1},x)
		\;\equiv_{\,O_n^\circ(u,V)\,}\;
		\frac{1}{(m-1)!}\,\frac{d^{\,m-1}}{dx^{\,m-1}}Y_L(u,x).
		\end{equation}
		\item Let $Y_L(u_{-m}\mathbf1,x)$ denote the left truncated field
		(Definition~\ref{def:YL}) of the homogeneous vector $u_{-m}\mathbf1$, not the derivative
		field of Convention~\ref{conv:descendant-derivative}.  Then
		\begin{equation}\label{eq:YL-descendants-derivative}
		Y_L(u_{-m}\mathbf{1},x)
		\;\equiv_{\,O_n^\circ(u,V)\,}\;
		\frac{1}{(m-1)!}\,\frac{d^{\,m-1}}{dx^{\,m-1}}Y_L(u,x).
		\end{equation}
		\item With $p_{j,m,u}(x)$ as in \eqref{eq:pjmu-def},
		\begin{equation}\label{eq:YL-descendants-p}
		\frac{1}{(m-1)!}\,\frac{d^{\,m-1}}{dx^{\,m-1}}Y_L(u,x)
		=
		\sum_{j\ge -2n-1} u_j\,p_{j,m,u}(x).
		\end{equation}
	\end{enumerate}
\end{lem}

\begin{proof}
	\textup{(a)} By the creation and translation properties,
	$u_{-m}\mathbf1=\frac{1}{(m-1)!}L(-1)^{m-1}u$, and the $L(-1)$--derivative property gives
	$Y(u_{-m}\mathbf1,x)=\frac{1}{(m-1)!}\frac{d^{m-1}}{dx^{m-1}}Y(u,x)$.  Hence
	\[
	Y(u_{-m}\mathbf1,x)-\frac{1}{(m-1)!}\frac{d^{\,m-1}}{dx^{\,m-1}}Y_L(u,x)
	=
	-\frac{1}{(m-1)!}\frac{d^{\,m-1}}{dx^{\,m-1}}\Delta(u,x).
	\]
	Applied to $v\in V$, each coefficient of the right side is a scalar multiple of a coefficient of
	$\Delta(u,x)v$, which lies in $O_n^\circ(u,V)$ by
	Proposition~\ref{prop:YL-On-truncation}\textup{(3)}.

	\textup{(b)} By Proposition~\ref{prop:YL-On-truncation}\textup{(3)} applied to the homogeneous
	vector $u_{-m}\mathbf1$,
	\[
	Y_L(u_{-m}\mathbf1,x)\equiv_{O_n^\circ(u_{-m}\mathbf1,V)}Y(u_{-m}\mathbf1,x).
	\]
	Since $u_{-m}\mathbf1=\frac{1}{(m-1)!}L(-1)^{m-1}u$ and
	$O_n^\circ(\cdot,V)$ is unchanged when its first argument is multiplied by a nonzero scalar,
	$m-1$ applications of Lemma~\ref{lem:Onu-Lminus1-inclusion} give
	$O_n^\circ(u_{-m}\mathbf1,V)\subseteq O_n^\circ(u,V)$.  Combine with \textup{(a)}.

	\textup{(c)} For each $v\in V$ the sum $Y_L(u,x)v=\sum_{j\ge-2n-1}u_jv\,p_{j,u}(x)$ is finite,
	so it may be differentiated termwise, and \eqref{eq:pjmu-def} gives \eqref{eq:YL-descendants-p}.
\end{proof}

The congruence \eqref{eq:YL-descendants-derivative} compares the two possible meanings of
$Y_L(u_{-r}\mathbf1,x)$ only when this field acts last.  In iterated products we fix one of them.

\begin{conv}[Vacuum descendants in products]\label{conv:descendant-derivative}
	Let $u\in V$ be homogeneous and let $r\ge1$.  From now on, in every product of fields (in
	particular in Lemma~\ref{lem:one-step-general}, Proposition~\ref{prop:one-sided-iterate},
	and Section~\ref{sec:finite-length}), the symbol $Y_L(u_{-r}\mathbf1,x)$ denotes the
	derivative field
	\begin{equation}\label{eq:descendant-derivative-convention}
	Y_L(u_{-r}\mathbf1,x):=\frac{1}{(r-1)!}\,\partial_x^{\,r-1}Y_L(u,x),
	\end{equation}
	and $Y_L^{\pm}(u_{-r}\mathbf1,x)$ denote its regular and singular parts.  For $r=1$ this is
	$Y_L(u,x)$.  This is a choice of notation, not a congruence.  When $Y_L(u_{-r}\mathbf1,x)$ is
	the leftmost factor, the derivative field and the truncated field of the vector
	$u_{-r}\mathbf1$ are congruent modulo $O_n^\circ(u,V)$ by
	\eqref{eq:YL-descendants-derivative}; when it is not the leftmost factor, the two products need
	not be congruent, and only the derivative field is meant.
\end{conv}

Since $\partial_x$ maps $\C[[x]]$ to itself and maps $x^{-j-1}$ ($j\ge0$) to a multiple of
$x^{-j-2}$, it commutes with taking regular and singular parts:
$\bigl(\partial_x^{r-1}F(x)\bigr)^\pm=\partial_x^{r-1}F^\pm(x)$.  In particular, for the
derivative field of Convention~\ref{conv:descendant-derivative},
$Y_L^\pm(u_{-r}\mathbf1,x)=\frac{1}{(r-1)!}\partial_x^{r-1}Y_L^\pm(u,x)$ and
$Y_L^-(u_{-r}\mathbf1,x)=Y^-(u_{-r}\mathbf1,x)$.

\subsection{Products and commutators with the regular part of \texorpdfstring{$Y_L$}{YL}}

\begin{lem}[Products with $Y_L$]\label{lem:Yprime-plus-product}
	\begin{enumerate}[label=\textup{(\alph*)}]
		\item Let $b_1,\dots,b_k\in V$ be homogeneous, and let each $G_\ell(x)$ be one of
		$Y_L(b_\ell,x)$, $Y_L^{+}(b_\ell,x)$, $Y_L^{-}(b_\ell,x)$, or $\frac{1}{(r-1)!}\partial_x^{r-1}$
		applied to one of these, for some $r\ge1$.  Then for every $z\in V$, the vector
		$G_1(x)\cdots G_k(x)z$ is a finite sum of terms $f(x)y$ with $f(x)\in\C((x))$ and $y\in V$.
		Consequently $Y(a,x)G_1(x)\cdots G_k(x)z\in V((x))$ for every $a\in V$.
		\item Let $u\in V$ be homogeneous, $r\ge1$, and
		$G(x):=\frac{1}{(r-1)!}\partial_x^{r-1}Y_L^+(u,x)$.  For all $v,w\in V$,
		\[
		Y(v,x)\,G(x)\,w\in V((x)),
		\qquad
		G(x)\,Y(v,x)\,w\in V((x)).
		\]
		\item With $u,r,G$ as in \textup{(b)}, the products $G(x_1)Y(v,x_2)w$ and $Y(v,x_2)G(x_1)w$
		lie in $V[[x_1]]((x_2))$.  Setting $x_1=x_2=x$ is a well-defined linear map
		$V[[x_1]]((x_2))\to V((x))$; it sends these two products to $G(x)Y(v,x)w$ and $Y(v,x)G(x)w$,
		and it is compatible with multiplication by elements of $\C[[x_1]]((x_2))$.
	\end{enumerate}
\end{lem}

\begin{proof}
	\textup{(a)} By Definition~\ref{def:YL} and Proposition~\ref{prop:YL-On-truncation}, each $G(x)$
	of the listed types has the form $\sum_{j\ge-2n-1}b_j\,g_j(x)$, where $g_j(x)\in\C((x))$ is
	$p_{j,b}(x)$, its regular or singular part, or a derivative of one of these.  For $y\in V$
	only the finitely many $j\ge-2n-1$ with $b_jy\ne0$ contribute, so
	$G(x)\bigl(f(x)y\bigr)=\sum_jf(x)g_j(x)\,b_jy$ is a finite sum of the required form.  Induction
	on $k$, starting from $G_k(x)z$, proves the first claim, and then
	$Y(a,x)\sum_if_i(x)y_i=\sum_if_i(x)Y(a,x)y_i\in V((x))$.

	\textup{(b)} The first product is \textup{(a)} with $k=1$.  For the second, write
	$Y(v,x)w=\sum_{p\ge P}c_px^p$ and $G(x)=\sum_{t\ge0}A_tx^t$ with $A_t\in\End(V)$; only
	nonnegative powers occur by \eqref{eq:YL-plus-band-tail}.  The coefficient of $x^N$ in
	$\sum_{t\ge0}\sum_{p\ge P}A_tc_p\,x^{t+p}$ is the finite sum over $0\le t\le N-P$, and it
	vanishes for $N<P$.

	\textup{(c)} By \textup{(a)}, $G(x_1)w=\sum_jg_j(x_1)u_jw$ is a finite sum with
	$g_j(x_1)\in\C[[x_1]]$, so
	\[
	Y(v,x_2)G(x_1)w=\sum_jg_j(x_1)Y(v,x_2)u_jw\in V[[x_1]]((x_2)).
	\]
	With $Y(v,x_2)w=\sum_{p\ge P}c_px_2^p$, each $G(x_1)c_p$ lies in $V[[x_1]]$, so
	\[
	G(x_1)Y(v,x_2)w=\sum_{p\ge P}\bigl(G(x_1)c_p\bigr)x_2^p\in V[[x_1]]((x_2)).
	\]
	For
	$\sum_{a\ge0,\,b\ge B}c_{a,b}x_1^ax_2^b$, the coefficient of $x^N$ after the substitution is
	$\sum_{a+b=N}c_{a,b}$, a finite sum over $0\le a\le N-B$, which vanishes for $N<B$.  Hence the
	substitution is well defined and linear.  If $f\in\C[[x_1]]((x_2))$ and
	$g\in V[[x_1]]((x_2))$, then the coefficient of $x^N$ in the substitution of $fg$ and in
	$f(x,x)g(x,x)$ are both $\sum f_{a,b}g_{a',b'}$ over $a+a'+b+b'=N$, a finite sum since
	$a,a'\ge0$ and $b,b'$ are bounded below; so the substitution is multiplicative.  Comparing coefficients with \textup{(b)} identifies the
	images of the two products.
\end{proof}

We use the commutator formula \cite{LL}
\begin{equation}\label{eq:LL-comm}
\bigl[\,u^{(1)}_{\ell},Y(u^{(2)},x)\,\bigr]
=
\sum_{i\ge0}\binom{\ell}{i}x^{\ell-i}Y(u^{(1)}_iu^{(2)},x),
\qquad \ell\in\Z,
\end{equation}
where the sum is finite because $u^{(1)}_iu^{(2)}=0$ for $i\gg0$.  For $\ell\ge0$ the terms with
$i>\ell$ vanish.

For homogeneous $u\in V$, $i\ge0$, and $r\ge1$, define the \emph{regular commutator kernels}
\begin{equation}\label{eq:Fiu-plus-def}
F^+_{i,u}(x_1,x_2)
:=
\sum_{j=-2n-1}^{-1}\binom{j}{i}\,p_{j,u}(x_1)\,x_2^{\,j-i}
\;+\;
x_1^{2n+1}\sum_{j\ge0}\binom{j}{i}\,P_{j,u}(x_1)\,x_2^{\,j-i},
\end{equation}
\begin{equation}\label{eq:Fium-def}
F^+_{i,r,u}(x_1,x_2)
:=
\frac{1}{(r-1)!}\,\partial_{x_1}^{\,r-1} F^+_{i,u}(x_1,x_2).
\end{equation}
In the second sum of \eqref{eq:Fiu-plus-def} only $j\ge i$ contribute, so
$F^+_{i,u}(x_1,x_2)\in\C[[x_1]]((x_2))$.  For $j\ge0$, \eqref{eq:p-P-decomp} gives
$x^{2n+1}P_{j,u}(x)=p_{j,u}(x)-x^{-j-1}$.  Applying $\frac{1}{(r-1)!}\partial_x^{r-1}$ and using $\frac{1}{(r-1)!}\partial_x^{r-1}x^{-j-1}=\binom{-j-1}{r-1}x^{-j-r}$
and \eqref{eq:pm-Pm-decomp} gives
$\frac{1}{(r-1)!}\partial_x^{r-1}\bigl(x^{2n+1}P_{j,u}(x)\bigr)=x^{2n+2-r}P_{j,r,u}(x)$.  Hence
\begin{equation}\label{eq:Fium-explicit}
F^+_{i,r,u}(x_1,x_2)
=
\sum_{j=-2n-1}^{-1}\binom{j}{i}\,p_{j,r,u}(x_1)\,x_2^{\,j-i}
\;+\;
x_1^{2n+2-r}\sum_{j\ge0}\binom{j}{i}\,P_{j,r,u}(x_1)\,x_2^{\,j-i},
\end{equation}
with $p_{j,r,u}$ and $P_{j,r,u}$ as in Lemma~\ref{lem:p-higher-derivatives}, and
$F^+_{i,1,u}=F^+_{i,u}$.  Since $\partial_{x_1}$ preserves $\C[[x_1]]((x_2))$, also
$F^+_{i,r,u}(x_1,x_2)\in\C[[x_1]]((x_2))$; in particular, for $r>2n+2$ the product
$x_1^{2n+2-r}P_{j,r,u}(x_1)$ is a power series, being a derivative of one.

\begin{lem}[Commutators with the regular part of $Y_L$]\label{lem:comm-Yprime-plus}\label{lem:comm-Yprime-plus-deriv}
Let $u,v\in V$ be homogeneous and let $r\ge1$.  Then in
$\Hom\!\bigl(V,V[[x_1]]((x_2))\bigr)$,
\begin{align}
\biggl[\frac{\partial_{x_1}^{r-1}Y_L^+(u,x_1)}{(r-1)!},Y(v,x_2)\biggr]
&=\sum_{i\ge0}F^+_{i,r,u}(x_1,x_2)Y(u_iv,x_2),
\label{eq:comm-Yprime-plus-deriv-FY}\\
&\equiv_{\sum_{i\ge0}O_n^\circ(u_iv,V)}
\sum_{i\ge0}F^+_{i,r,u}(x_1,x_2)Y_L(u_iv,x_2).
\label{eq:comm-Yprime-plus-deriv-F}
\end{align}
After setting $x_1=x_2=x$, in $\Hom(V,V((x)))$,
\begin{align}
\biggl[\frac{\partial_x^{r-1}Y_L^+(u,x)}{(r-1)!},Y(v,x)\biggr]
&=\sum_{i\ge0}F^+_{i,r,u}(x,x)Y(u_iv,x),
\label{eq:comm-Yprime-plus-deriv-diagonal-FY}\\
&\equiv_{\sum_{i\ge0}O_n^\circ(u_iv,V)}
\sum_{i\ge0}F^+_{i,r,u}(x,x)Y_L(u_iv,x).
\label{eq:comm-Yprime-plus-deriv-diagonal}
\end{align}
\end{lem}

\begin{proof}
	By \eqref{eq:YL-plus-band-tail} and the computation preceding \eqref{eq:Fium-explicit},
	\[
	\frac{1}{(r-1)!}\partial_{x_1}^{r-1}Y_L^+(u,x_1)
	=\sum_{j\ge-2n-1}u_j\,g_j(x_1),
	\]
	where
	\[
	g_j(x_1)=
	\begin{cases}
	p_{j,r,u}(x_1),&-2n-1\le j\le-1,\\
	x_1^{2n+2-r}P_{j,r,u}(x_1),&j\ge0,
	\end{cases}
	\]
	with $g_j(x_1)\in\C[[x_1]]$.  Fix $w\in V$.  Both products in the commutator, applied to $w$,
	lie in $V[[x_1]]((x_2))$ by Lemma~\ref{lem:Yprime-plus-product}\textup{(c)}, and for each fixed
	power of $x_2$ only finitely many $j$ contribute (since $Y(v,x_2)w$ and $Y(v,x_2)u_jw$ are lower
	truncated in $x_2$, and $u_jc=0$ for $j\gg0$).  Hence the commutator may be computed termwise
	in $j$, and \eqref{eq:LL-comm} with $\ell=j$ gives
	\[
	\sum_{j\ge-2n-1}g_j(x_1)\sum_{i\ge0}\binom{j}{i}x_2^{\,j-i}\,Y(u_iv,x_2)w.
	\]
	The $i$-sum is finite, uniformly in $j$, since $u_iv=0$ for $i\gg0$.  For fixed $i$ and a fixed
	power of $x_2$, only finitely many $j$ contribute, because $j\ge-2n-1$ and $Y(u_iv,x_2)w$ is
	lower truncated.  Interchanging the two sums and using \eqref{eq:Fium-explicit} gives
	\eqref{eq:comm-Yprime-plus-deriv-FY}.

	For \eqref{eq:comm-Yprime-plus-deriv-F}, Proposition~\ref{prop:YL-On-truncation}\textup{(3)}
	applied to the homogeneous vectors $u_iv$ gives
	\[
	Y(u_iv,x_2)\equiv_{O_n^\circ(u_iv,V)}Y_L(u_iv,x_2);
	\]
	multiply by the scalar series
	$F^+_{i,r,u}(x_1,x_2)$ and sum over the finitely many $i$ with $u_iv\ne0$.  The diagonal
	identities follow by applying the substitution of
	Lemma~\ref{lem:Yprime-plus-product}\textup{(c)}, which is linear, compatible with
	multiplication by $F^+_{i,r,u}(x_1,x_2)\in\C[[x_1]]((x_2))$, and sends each coefficient to a
	finite sum of coefficients, so it preserves congruences.
\end{proof}

\begin{lem}[Explicit diagonal kernel]\label{lem:Cium-YL}
	Let $u\in V$ be homogeneous, $i\ge0$, and $r\ge1$.  The diagonal kernel admits the form
	\begin{equation}\label{eq:Fium-plus-diag} \begin{aligned} F_{i,r,u}^{+}(x,x)  &= \sum_{j=-2n-1}^{-1}\binom{j}{i}\,p_{j,r,u}(x)\,x^{\,j-i} \;+\; x^{2n+2-r-i}\sum_{k\ge0}\binom{k}{i}\,P_{k,r,u}(x)\,x^{k} \\[2pt] &= x^{-i-r}(1+x)^{-(\wt u+n)-(r-1)}\,Q_{i,r,u}(x) \;+\; x^{2n+2-r-i}\sum_{k\ge0}\binom{k}{i}\,P_{k,r,u}(x)\,x^{k}. \end{aligned} \end{equation}
	Here \(p_{j,r,u}(x)\) and \(P_{k,r,u}(x)\) are as in
	Lemma~\ref{lem:p-higher-derivatives}, and \(Q_{i,r,u}(x)\in\C[x]\) is the polynomial
	\begin{equation}\label{eq:Qium-def}
	\begin{aligned}
	Q_{i,r,u}(x)
	&:=
	\sum_{\ell=0}^{r-1}
	(-1)^{r-1-\ell}
	\binom{\wt u + n + r - 2 - \ell}{\,r-1-\ell\,}
	(1+x)^{\ell}
	\\
	&\quad\cdot
	\sum_{s=0}^{2n}
	\Biggl(
	\sum_{k=s-2n-1}^{-1}
	\binom{k}{i}\,\binom{-k-1+s}{\ell}
	\Biggr)
	\binom{\wt u + n}{s}\,x^{\,s-\ell+r-1}.
	\end{aligned}
	\end{equation}
	In particular, in Cases~{\rm(2)} and~{\rm(3a)} of Remark~\ref{rem:YL-cases} we have
	$P_{j,u}(x)=0$ for all $j\ge0$, hence $P_{k,r,u}(x)=0$ for all $k\ge0$ and the second term
	in \eqref{eq:Fium-plus-diag} vanishes.
\end{lem}

\begin{proof}
	The first line of \eqref{eq:Fium-plus-diag} is \eqref{eq:Fium-explicit} at $x_1=x_2=x$.  It
	remains to rewrite the finite \(j<0\) part.  By Lemma~\ref{lem:p-higher-derivatives},
	equation \eqref{eq:pm-factor-form} (with its summation indices renamed \(i\to s\) and
	\(s\to\ell\), since \(i\) is the kernel index here),
	\[
	\begin{aligned}
	&\sum_{j=-2n-1}^{-1}\binom{j}{i}p_{j,r,u}(x)x^{j-i}
	= (1+x)^{-(\wt u+n+r-1)}  \\& \cdot
	\sum_{\ell=0}^{r-1}
	\binom{-\wt u-n}{r-1-\ell}(1+x)^\ell
	\sum_{j=-2n-1}^{-1}\binom{j}{i}
	\sum_{s=0}^{2n+j+1}
	\binom{\wt u+n}{s}\binom{s-j-1}{\ell}x^{s-i-1-\ell}.
	\end{aligned}
	\]
	Interchanging the finite \(j\)- and \(s\)-sums, the condition
	\(0\le s\le2n+j+1\) becomes \(s-2n-1\le j\le -1\).  Therefore
	\[
	\begin{aligned}
	&\sum_{j=-2n-1}^{-1}\binom{j}{i}p_{j,r,u}(x)x^{j-i}  \\
	&= x^{-i-r}(1+x)^{-(\wt u+n+r-1)}
	\sum_{\ell=0}^{r-1}
	(-1)^{r-1-\ell}
	\binom{\wt u+n+r-2-\ell}{r-1-\ell}(1+x)^\ell  \\
	&\qquad\cdot
	\sum_{s=0}^{2n}
	\left(
	\sum_{j=s-2n-1}^{-1}
	\binom{j}{i}\binom{s-j-1}{\ell}
	\right)
	\binom{\wt u+n}{s}x^{s-\ell+r-1},
	\end{aligned}
	\]
	where
	\[
	\binom{-\wt u-n}{r-1-\ell}
	=
	(-1)^{r-1-\ell}
	\binom{\wt u+n+r-2-\ell}{r-1-\ell}.
	\]
	By \eqref{eq:Qium-def}, the last display is
	\[
	x^{-i-r}(1+x)^{-(\wt u+n+r-1)}Q_{i,r,u}(x).
	\]
	The expression \eqref{eq:Qium-def} is a polynomial, since all sums are finite
	and \(s-\ell+r-1\ge0\).  Substituting this expression for the finite \(j<0\)
	part gives \eqref{eq:Fium-plus-diag}.
\end{proof}

\subsection{The one-step identity and one-sided iterates}

\begin{lem}[One-step truncated identity]\label{lem:one-step-general}\label{cor:one-step-YL-minimal}
Let $u,v\in V$ be homogeneous, let $r\ge1$, and put
\[
U_L=O_n^\circ(u,V)+O_n^\circ(v,V)+\sum_{i\ge0}O_n^\circ(u_iv,V).
\]
Here $Y_L(u_{-r}\mathbf1,x)$ is the derivative field of
Convention~\ref{conv:descendant-derivative}.  In $\Hom(V,V((x)))$,
\begin{equation}\label{eq:one-step-YL-U}
Y(u_{-r}v,x)\equiv_{U_L}
\sum_{i\ge0}F^+_{i,r,u}(x,x)Y_L(u_iv,x)
 +Y_L(v,x)Y_L(u_{-r}\mathbf1,x).
\end{equation}
\end{lem}
\begin{proof}
First we prove the smaller-modulus identity
\begin{equation}\label{eq:one-step-YL-expanded}
	Y(u_{-r}v,x)
	\;\equiv_{\,O_n^\circ(u,V)\,}\;
	\sum_{i\ge0} F^+_{i,r,u}(x,x)\,Y(u_i v,x)\;+\;Y(v,x)\,Y_L(u_{-r}\mathbf1,x).
	\end{equation}
	Fix $z\in V$.  All products below are applied to $z$, and they are well defined by
	Lemma~\ref{lem:Yprime-plus-product}.  By the $L(-1)$--derivative property,
	$Y(u_{-r}\mathbf1,x)=\frac{1}{(r-1)!}\partial_x^{r-1}Y(u,x)$, and $\partial_x$ commutes with
	taking regular and singular parts, so
	$Y^{\pm}(u_{-r}\mathbf1,x)=\frac{1}{(r-1)!}\partial_x^{r-1}Y^{\pm}(u,x)$.  By
	Proposition~\ref{prop:YL-On-truncation}\textup{(2)}, $Y_L^-(u,x)=Y^-(u,x)$ and
	$Y_L^+(u,x)=Y^+(u,x)+\Delta(u,x)$.  Finally $u_{-r}v=(u_{-r}\mathbf1)_{-1}v$, so the
	normal-ordered product formula $Y(a_{-1}b,x)=Y^+(a,x)Y(b,x)+Y(b,x)Y^-(a,x)$ \cite{LL}
	applies with $a=u_{-r}\mathbf1$ and $b=v$.  Therefore
	\begin{align*}
	Y(u_{-r}v,x)
	&=Y^+(u_{-r}\mathbf1,x)\,Y(v,x)+Y(v,x)\,Y^-(u_{-r}\mathbf1,x)\\
	&=\tfrac{1}{(r-1)!}\partial_x^{r-1}Y^+(u,x)\,Y(v,x)
	+Y(v,x)\,\tfrac{1}{(r-1)!}\partial_x^{r-1}Y_L^-(u,x)\\
	&\equiv_{O_n^\circ(u,V)}\tfrac{1}{(r-1)!}\partial_x^{r-1}Y_L^+(u,x)\,Y(v,x)
	+Y(v,x)\,\tfrac{1}{(r-1)!}\partial_x^{r-1}Y_L^-(u,x)\\
	&=\Bigl[\tfrac{1}{(r-1)!}\partial_x^{r-1}Y_L^+(u,x),\,Y(v,x)\Bigr]
	+Y(v,x)\,\tfrac{1}{(r-1)!}\partial_x^{r-1}Y_L(u,x)\\
	&=\sum_{i\ge0}F^+_{i,r,u}(x,x)\,Y(u_iv,x)+Y(v,x)\,Y_L(u_{-r}\mathbf1,x).
	\end{align*}
	The third line is the only replacement, and it is a leftmost replacement.  The difference of
	the third and second lines is $\frac{1}{(r-1)!}\bigl(\partial_x^{r-1}\Delta(u,x)\bigr)Y(v,x)z$.
	Write $Y(v,x)z=\sum_{p\ge P}c_px^p$.  Since $\Delta(u,x)$ has only nonnegative powers of $x$
	(Proposition~\ref{prop:YL-On-truncation}\textup{(3)}), the coefficient of $x^M$ in this
	difference is a finite sum, over $P\le p\le M$, of coefficients of
	$\partial_x^{r-1}\Delta(u,x)c_p$; each is a scalar multiple of a coefficient of
	$\Delta(u,x)c_p$, hence lies in $O_n^\circ(u,V)$.  The fourth line is $AB=[A,B]+BA$ together
	with $Y_L^++Y_L^-=Y_L$.  The fifth line is \eqref{eq:comm-Yprime-plus-deriv-diagonal-FY}
	together with \eqref{eq:descendant-derivative-convention}.  This proves
	\eqref{eq:one-step-YL-expanded}.  The truncated field of the vector $u_{-r}\mathbf1$ is never
	used.

Put $U_L^*=O_n^\circ(u,V)+\sum_{i\ge0}O_n^\circ(u_iv,V)$.  The intermediate identity is
\begin{equation}\label{eq:one-step-YL-Ustar}
	Y(u_{-r}v,x)
	\;\equiv_{\,U_L^*\,}\;
	\sum_{i\ge0} F^+_{i,r,u}(x,x)\,Y_L(u_i v,x)\;+\;Y(v,x)\,Y_L(u_{-r}\mathbf1,x).
\end{equation}
	Indeed, $u_iv=0$ for $i\gg0$, each
	$F^+_{i,r,u}(x,x)$ is a scalar series, and
	$Y(u_iv,x)z\equiv_{O_n^\circ(u_iv,V)}Y_L(u_iv,x)z$ by
	Proposition~\ref{prop:YL-On-truncation}\textup{(3)}.  Finally, since $U_L=U_L^*+O_n^\circ(v,V)$,
	for \eqref{eq:one-step-YL-U} replace
	the leftmost factor $Y(v,x)$ in $Y(v,x)Y_L(u_{-r}\mathbf1,x)z$: by
	Lemma~\ref{lem:Yprime-plus-product}\textup{(a)} the vector $Y_L(u_{-r}\mathbf1,x)z$ is a finite
	sum of terms $f(x)y$ with $f(x)\in\C((x))$ and $y\in V$, and
	$Y(v,x)y\equiv_{O_n^\circ(v,V)}Y_L(v,x)y$.
\end{proof}

\begin{prop}[One-sided iterated expansion]\label{prop:one-sided-iterate}\label{cor:one-sided-iterate-YL}
	Let $\ell\ge2$, let $u^{(1)},\dots,u^{(\ell)}\in V$ be homogeneous, and let
	$r_1,\dots,r_\ell\ge1$.  Set
	\[
	w^{(s)}:=u^{(s)}_{-r_s}u^{(s+1)}_{-r_{s+1}}\cdots u^{(\ell)}_{-r_\ell}\mathbf1,
	\qquad w:=w^{(1)}.
	\]
	Each $Y_L(u^{(s)}_{-r_s}\mathbf1,x)$ below is the derivative field of
	Convention~\ref{conv:descendant-derivative}.  Set
	\[
	\mathcal U_\ell
	:=
	\sum_{s=1}^{\ell}O_n^\circ(u^{(s)},V)
	+
	\sum_{s=1}^{\ell-1}\ \sum_{i\ge0} O_n^\circ\!\bigl(u^{(s)}_i w^{(s+1)},V\bigr).
	\]
	Then in $\Hom(V,V((x)))$,
		\begin{equation}\label{eq:general-truncated-iterate-onesided-YL}
	\begin{aligned}
	Y(w,x)\;\equiv_{\mathcal U_\ell}\;&
	\sum_{i\ge0}F^+_{i,r_1,u^{(1)}}(x,x)\,
	Y_L\!\bigl(u^{(1)}_i w^{(2)},x\bigr)
	\\
	&+\;
	\sum_{s=2}^{\ell-1}\ \sum_{i\ge0}
	F^+_{i,r_s,u^{(s)}}(x,x)\,
	Y_L\!\bigl(u^{(s)}_i w^{(s+1)},x\bigr)\,
	Y_L\!\bigl(u^{(s-1)}_{-r_{s-1}}\mathbf1,x\bigr)\cdots
	Y_L\!\bigl(u^{(1)}_{-r_1}\mathbf1,x\bigr)
	\\
	&+\;
	Y_L\!\bigl(u^{(\ell)}_{-r_\ell}\mathbf1,x\bigr)\cdots
	Y_L\!\bigl(u^{(1)}_{-r_1}\mathbf1,x\bigr).
	\end{aligned}
	\end{equation}
\end{prop}

\begin{proof}
Put $\mathcal U_\ell^{(0)}=\sum_{s=1}^{\ell}O_n^\circ(u^{(s)},V)$.
We first prove the smaller-modulus identity
\begin{equation}\label{eq:general-truncated-iterate-onesided-YL-Y}
	\begin{aligned}
	Y(w,x)\;\equiv_{\mathcal U_\ell^{(0)}}\;&
	\sum_{i\ge0}F^+_{i,r_1,u^{(1)}}(x,x)\,
	Y\!\bigl(u^{(1)}_i w^{(2)},x\bigr)
	\\
	&+\;
	\sum_{s=2}^{\ell-1}\ \sum_{i\ge0}
	F^+_{i,r_s,u^{(s)}}(x,x)\,
	Y\!\bigl(u^{(s)}_i w^{(s+1)},x\bigr)\,
	Y_L\!\bigl(u^{(s-1)}_{-r_{s-1}}\mathbf1,x\bigr)\cdots
	Y_L\!\bigl(u^{(1)}_{-r_1}\mathbf1,x\bigr)
	\\
	&+\;
	Y_L\!\bigl(u^{(\ell)}_{-r_\ell}\mathbf1,x\bigr)\cdots
	Y_L\!\bigl(u^{(1)}_{-r_1}\mathbf1,x\bigr).
	\end{aligned}
	\end{equation}

	Write $\Pi_s(x):=Y_L\!\bigl(u^{(s)}_{-r_s}\mathbf1,x\bigr)\cdots
	Y_L\!\bigl(u^{(1)}_{-r_1}\mathbf1,x\bigr)$ for $s\ge1$.  Fix $z\in V$; all products are applied
	to $z$ and are well defined by Lemma~\ref{lem:Yprime-plus-product}\textup{(a)}.

	We prove \eqref{eq:general-truncated-iterate-onesided-YL-Y} by induction on $\ell$.  For
	$\ell=2$, apply \eqref{eq:one-step-YL-expanded} to $Y(u^{(1)}_{-r_1}w^{(2)},x)$ with
	$w^{(2)}=u^{(2)}_{-r_2}\mathbf1$, modulo $O_n^\circ(u^{(1)},V)$.  In the product term
	$Y(w^{(2)},x)\Pi_1(x)z$, the vector $\Pi_1(x)z$ is a finite sum of terms $f(x)y$, and
	\eqref{eq:Y-descendants-via-YL-derivative} replaces the leftmost factor $Y(w^{(2)},x)$ by the
	derivative field $Y_L(u^{(2)}_{-r_2}\mathbf1,x)$ modulo $O_n^\circ(u^{(2)},V)$.  This gives
	\eqref{eq:general-truncated-iterate-onesided-YL-Y} for $\ell=2$.

	Let $\ell\ge3$ and assume \eqref{eq:general-truncated-iterate-onesided-YL-Y} for words of length
	$\ell-1$.  By \eqref{eq:one-step-YL-expanded},
	\[
	Y(w^{(1)},x)\equiv_{O_n^\circ(u^{(1)},V)}
	\sum_{i\ge0}F^+_{i,r_1,u^{(1)}}(x,x)\,Y(u^{(1)}_i w^{(2)},x)
	\;+\;Y(w^{(2)},x)\,\Pi_1(x).
	\]
	Apply the induction hypothesis, for the word $w^{(2)}$ built from $u^{(2)},\dots,u^{(\ell)}$, to
	the leftmost factor $Y(w^{(2)},x)$ of $Y(w^{(2)},x)\Pi_1(x)z$.  This is a leftmost replacement:
	$\Pi_1(x)z$ is a finite sum of terms $f(x)y$ with $f(x)\in\C((x))$ and $y\in V$, and the
	hypothesis applies to each $y$ modulo $\sum_{s=2}^{\ell}O_n^\circ(u^{(s)},V)$.  Multiplying the
	result on the right by $\Pi_1(x)$ produces the terms $s=2,\dots,\ell-1$ and the final product of
	\eqref{eq:general-truncated-iterate-onesided-YL-Y}, with total modulus $\mathcal U_\ell^{(0)}$.

	Finally, \eqref{eq:general-truncated-iterate-onesided-YL} follows from
	\eqref{eq:general-truncated-iterate-onesided-YL-Y}: in each term, $Y(u^{(s)}_iw^{(s+1)},x)$ is
	the leftmost factor, $\Pi_{s-1}(x)z$ (or $z$ itself for $s=1$) is a finite sum of terms $f(x)y$,
	and, by Proposition~\ref{prop:YL-On-truncation}\textup{(3)},
	\[
	Y(u^{(s)}_iw^{(s+1)},x)y\equiv_{O_n^\circ(u^{(s)}_iw^{(s+1)},V)}Y_L(u^{(s)}_iw^{(s+1)},x)y.
	\]  Summing over $s$ and the finitely many
	nonzero $u^{(s)}_iw^{(s+1)}$ gives the modulus $\mathcal U_\ell$.
\end{proof}

\section{Replacement residues and finite-length circle subspaces}\label{sec:finite-length}

Throughout this section we fix $n\in\N$.

 \begin{defn}[Strong generating set]\label{def:strong-generating-set}
 	Let \(V\) be a vertex operator algebra and let \(S\subseteq V\) be a subset.
 	Write \(S=\{u^{(i)}\}_{i\in I}\).
 	We say that \(S\) \emph{strongly generates} \(V\) if \(V\) is linearly spanned by vectors of the form
 	\begin{equation}\notag
 	u^{(i_1)}_{-m_1}\,u^{(i_2)}_{-m_2}\cdots u^{(i_r)}_{-m_r}\,\mathbf{1},
 	\end{equation}
 	where \(r\ge 0\), \(i_j\in I\), and \(m_j\in\Z_{\ge 1}\) for all \(j\).
 \end{defn}

\subsection{Replacement residues}

\begin{defn}[Residues attached to a replacement field]\label{def:Rmk-replacement}
	Fix $n\in\N$ and let $U\subseteq O_n^\circ(V)$ be a subspace.
	Let $v\in V$ be homogeneous (e.g.\ $v=s^{(1)}_{-r_1}\cdots s^{(\ell)}_{-r_\ell}\mathbf1$ with
	$s^{(j)}$ homogeneous), and let
	\[
	\widehat Y(v,x)\in\Hom\!\bigl(V,V((x))\bigr)
	\]
	be any formal series such that
	\begin{equation}\label{eq:Y-replaced-mod-U}
	Y(v,x)\ \equiv_U\ \widehat Y(v,x)
	\qquad\text{in }\Hom\!\bigl(V,V((x))\bigr)
	\end{equation}
	in the sense of Definition~\ref{def:W-congruence}\textup{(3)}. Applying the same residue kernel to \eqref{eq:Y-replaced-mod-U} turns a field-level replacement congruence
	into congruences between circle-type generators, producing $U$--reduced relations in $O_n^\circ(V)$.
	For integers $m,k\in\N$ with $m\ge k\ge0$ and $z\in V$, define
	\begin{equation}\label{eq:Rmk-replacement-def}
	v\,R^{n,\widehat Y}_{m,k}\,z
	\;:=\;
	\Res_x \frac{(1+x)^{\wt v+n+k}}{x^{m+2n+2}}\;\widehat Y(v,x)z.
	\end{equation}
	When $\widehat Y(v,x)=Y(v,x)$ this recovers $v\,R^n_{m,k}\,z$ from \eqref{eq:Rmk-def}.
\end{defn}
\begin{prop}[Replacement residues lie in $O_n^\circ(V)$]\label{prop:Rmk-replacement-in-On}
	Keep the notation of Definition~\ref{def:Rmk-replacement}.
	Then for all $m\ge k\ge0$ and $z\in V$ one has
	\begin{equation}\label{eq:Rmk-replacement-diff}
	v\,R^{n,\widehat Y}_{m,k}\,z
	\;\equiv_U\;
	v\,R^n_{m,k}\,z.
	\end{equation}
	Consequently,
	\begin{equation}\label{eq:Rmk-replacement-membership}
	v\,R^{n,\widehat Y}_{m,k}\,z\in O_n^\circ(V).
	\end{equation}
\end{prop}

\begin{proof}
	By \eqref{eq:Y-replaced-mod-U} we have
	$\bigl(\widehat Y(v,x)-Y(v,x)\bigr)z\in U((x))$ for all $z\in V$.
	Multiplying by $\frac{(1+x)^{\wt v+n+k}}{x^{m+2n+2}}\in\C((x))$ preserves membership in
	$U((x))$, hence
	\[
	\Res_x \frac{(1+x)^{\wt v+n+k}}{x^{m+2n+2}}\bigl(\widehat Y(v,x)-Y(v,x)\bigr)z \in U,
	\]
	which is exactly \eqref{eq:Rmk-replacement-diff} by \eqref{eq:Rmk-def} and
	\eqref{eq:Rmk-replacement-def}.

	By Lemma~\ref{Zhu-lemma}, $v\,R^n_{m,k}\,z\in O_n^\circ(V)$ for all $m\ge k\ge0$ and $z\in V$.
	Since $U\subseteq O_n^\circ(V)$, \eqref{eq:Rmk-replacement-diff} implies
	$v\,R^{n,\widehat Y}_{m,k}\,z\in O_n^\circ(V)$, proving \eqref{eq:Rmk-replacement-membership}.
\end{proof}

\begin{rem}[Using replacement residues]
\label{rem:what-replacement-residue-buys}
The calculation has two steps. First use leftmost replacement and the
truncated iterate identities to obtain $Y(w,x)\equiv_U\widehat Y(w,x)$.
Then apply the residue kernel in \eqref{eq:Rmk-replacement-def}.
Proposition~\ref{prop:Rmk-replacement-in-On} gives the resulting
residue congruence and shows that the replacement residues belong to
$O_n^\circ(V)$ whenever $U\subseteq O_n^\circ(V)$.
\end{rem}

\subsection{Finite-length circle subspaces and conjectures}

\begin{defn}[Finite-length circle-relation subspaces]\label{def:OnN-SV}
Fix $n\in\N$ and a finite homogeneous subset $S\subset V$.
For $N\ge1$, let $\mathcal P_{\le N}(S)$ consist of the vacuum words
\begin{equation}\label{eq:finite-length-vacuum-words}
w=s^{(1)}_{-m_1}\cdots s^{(r)}_{-m_r}\mathbf1,
\qquad 1\le r\le N,\quad s^{(j)}\in S,\quad
m_1,\ldots,m_{r-1}\in\Z,\quad m_r\ge1.
\end{equation}
Define
\begin{equation}\label{eq:OnNSV-def}
O_n^\circ(N,S,V)
:=\Span_\C\{wR^n_{m,k}z:\ w\in\mathcal P_{\le N}(S),\ z\in V,
\ m\ge k\ge0\}\subseteq O_n^\circ(V).
\end{equation}
The containment follows from Lemma~\ref{Zhu-lemma}.
The plus subspace is
\begin{equation}\label{eq:OnNSV-plus-def}
O_n^\circ(N,S,V)^+
:=\Span_\C\{wR^n_{m,k}z:\ z\in V,\ m\ge k\ge0,
\ w\text{ as specified below}\},
\end{equation}
where $w$ either has length at most $N-1$, or has length $N$ in
\eqref{eq:finite-length-vacuum-words} with $m_j\le0$ for at least one
$j<N$. Thus
\begin{equation}\label{eq:circle-length-inclusions}
O_n^\circ(N-1,S,V)\subseteq O_n^\circ(N,S,V)^+
\subseteq O_n^\circ(N,S,V)\qquad(N\ge2).
\end{equation}
If $S$ strongly generates $V$, then
\begin{equation}\label{eq:circle-length-exhaustion}
\bigcup_{N\ge1}O_n^\circ(N,S,V)=O_n^\circ(V),
\end{equation}
because every vector of $V$ is a finite linear combination of regular
vacuum words and the vacuum slice is zero.
\end{defn}

We expect the first correction term in Proposition~\ref{prop:one-sided-iterate} to have
residues in $O_n^\circ(V)$.  Precisely, let $u^{(1)},v\in V$ be homogeneous, let $r_1\ge1$, let
$m\ge k\ge0$, and let $z\in V$.  The expected inclusion is
\begin{equation}\label{eq:conj-first-term-residue}
\begin{aligned}
\Res_x\;&\frac{(1+x)^{\wt u^{(1)}+\wt v+r_1-1+n+k}}{x^{m+2n+2}}
\left(\sum_{i\ge0}F^+_{i,r_1,u^{(1)}}(x,x)\,
Y(u^{(1)}_i v,x)\right)z\\
&\in\ \sum_{i\ge0}O_n^\circ(u^{(1)}_iv,V)\ \subseteq\ O_n^\circ(V).
\end{aligned}
\end{equation}
For words in $S$ we state the following termwise refinement as a conjecture.

\begin{conj}[Correction residues in the plus subspace]
\label{conj:finite-level-revised}\label{conj:correction-residues-plus}
Fix $n\in\N$ and a finite homogeneous strong generating set $S$.
Let $N\ge2$, let $\underline s=(s^{(1)},\ldots,s^{(N)})\in S^N$,
and take $r_1,\ldots,r_N\ge1$.  Put
\[
w^{(q)}:=s^{(q)}_{-r_q}\cdots s^{(N)}_{-r_N}\mathbf1
\quad(1\le q\le N),\qquad w:=w^{(1)}.
\]
For every $m\ge k\ge0$, $z\in V$, and $0\le t\le N-2$, the
correction residue with $t$ trailing factors, whose singular field sits
at position $t+1$ of $w$, separately lies in the global plus
subspace:
\begin{equation}\label{eq:conj-termwise-correction-mod-plus}
\begin{aligned}
&\Res_x\frac{(1+x)^{\wt w+n+k}}{x^{m+2n+2}}
\sum_{i\ge0}F^+_{i,r_{t+1},s^{(t+1)}}(x,x)\,
Y\!\bigl(s^{(t+1)}_i w^{(t+2)},x\bigr)\\
&\qquad\cdot
Y_L(s^{(t)}_{-r_{t}}\mathbf1,x)\cdots
Y_L(s^{(1)}_{-r_1}\mathbf1,x)z
\in O_n^\circ(N,S,V)^+.
\end{aligned}
\end{equation}
For $t=0$ the product of $Y_L$ factors is absent. With $\ell=N$ and
$u^{(q)}=s^{(q)}$, the residue in \eqref{eq:conj-termwise-correction-mod-plus}
is that of the correction term in
\eqref{eq:general-truncated-iterate-onesided-YL} indexed by $t+1$, with
$Y_L$ replaced by $Y$ in its leftmost factor.
\end{conj}
\begin{defn}[Auxiliary-degree hypothesis]\label{def:auxiliary-degree}
Let $S$ be a finite homogeneous strong generating set of $V$.
We say that $S$ satisfies the auxiliary-degree hypothesis if there is
a function $d:S\to\Z_{\ge0}$ such that every singular product
$a_jb$, $a,b\in S$, $j\ge0$, has a finite expansion
\begin{equation}\label{eq:auxiliary-degree-expansion}
a_jb=\sum_\rho c_\rho\,
t^{(\rho,1)}_{-r_{\rho,1}}\cdots
t^{(\rho,\ell_\rho)}_{-r_{\rho,\ell_\rho}}\mathbf1,
\qquad t^{(\rho,\mu)}\in S,\quad r_{\rho,\mu}\ge1,
\end{equation}
in which each nonzero summand has the conformal weight of $a_jb$ and
satisfies
\begin{equation}\label{eq:auxiliary-degree-drop}
\sum_{\mu=1}^{\ell_\rho}d(t^{(\rho,\mu)})\le d(a)+d(b)-1.
\end{equation}
The empty word is allowed and has degree zero. The auxiliary degrees
do not replace the conformal weights in the residue kernels.
\end{defn}

\begin{rem}[Relation to good increasing filtrations]\label{rem:auxiliary-degree-Li}
For positive assigned degrees and a compatible decomposition of the
generating subspace, this is the generator condition in Li's
good increasing filtration construction \cite[Theorem~4.6]{Li}.
Here degree-zero generators are also allowed. The hypothesis concerns
expansions in the specified set $S$; the existence of an unrelated good
filtration does not by itself verify \eqref{eq:auxiliary-degree-drop}.
If all generators have nonnegative conformal weight, one can take
$d(s)=\wt s$: the weight of a regular word is the sum of the generator
weights plus $\sum_\mu(r_{\rho,\mu}-1)$, whereas
$\wt(a_jb)=\wt a+\wt b-j-1$.
\end{rem}

By Remark~\ref{rem:auxiliary-degree-Li}, the auxiliary-degree hypothesis holds, with
$d(s)=\wt s$, whenever every element of $S$ has nonnegative weight.  In particular it holds
for every finite homogeneous strong generating set of an $\N$-graded vertex operator algebra.

\begin{conj}[Finite length under the auxiliary-degree hypothesis]
\label{conj:finite-length-zhu-relation}\label{conj:auxiliary-finite-length}
Let $n\ge1$ and let $S$ satisfy Definition~\ref{def:auxiliary-degree}.
Then
\begin{equation}\label{eq:auxiliary-finite-length}
O_n^\circ(V)=O_n^\circ(n+1,S,V).
\end{equation}
In particular,
\begin{equation}\label{eq:auxiliary-Zhu-finite-length}
O_n(V)=O^L(V)+O_n^\circ(n+1,S,V).
\end{equation}
\end{conj}

\begin{rem}[Unrestricted question and scope]\label{rem:finite-length-unrestricted}
By Remark~\ref{rem:auxiliary-degree-Li}, the auxiliary-degree hypothesis can fail only
when $S$ contains elements of negative weight.  For such $S$, whether
\eqref{eq:auxiliary-finite-length}, or the weaker equality
\eqref{eq:auxiliary-Zhu-finite-length}, holds without the hypothesis is a separate
question; the proof proposed in subsequent work uses \eqref{eq:auxiliary-degree-drop}.
\end{rem}

\subsection{First reductions}

We now record the first reduction results for the subspaces $O_n^\circ(N,S,V)$
(Definition~\ref{def:OnN-SV}).

\begin{lem}[An $N$-iterate expansion via translation]
	\label{lem:triple-iterate-derivative-expansion}
	Let $N\ge 2$.
	Let $u^{(1)},\dots,u^{(N)}\in V$.
	Fix integers $q_1,\dots,q_{N-1}\in\Z$ and fix $q_N\in\N$ with $q_N\ge 1$.
	Then
	\begin{equation}\label{eq:N-iterate-step1}
	\begin{aligned}
	u^{(1)}_{q_1}\cdots u^{(N-1)}_{q_{N-1}}\,u^{(N)}_{-q_N}\mathbf1
	&=
	\sum_{j=0}^{q_N-1}
	\binom{q_{N-1}}{j}\,
	\frac{1}{(q_N-1-j)!} \\
	&\qquad\cdot
	u^{(1)}_{q_1}\cdots u^{(N-2)}_{q_{N-2}}\,
	L(-1)^{q_N-1-j}\Bigl(u^{(N-1)}_{q_{N-1}-j}\,u^{(N)}\Bigr).
	\end{aligned}
	\end{equation}
	Moreover, commuting $u^{(1)}_{q_1},\dots,u^{(N-2)}_{q_{N-2}}$ past
	$L(-1)^{q_N-1-j}$ gives
	\begin{equation}\label{eq:N-iterate-step2}
	\begin{aligned}
	&u^{(1)}_{q_1}\cdots u^{(N-1)}_{q_{N-1}}\,u^{(N)}_{-q_N}\mathbf1 \\
	&\qquad=
	\sum_{j=0}^{q_N-1}\;
	\sum_{\substack{i_1,\dots,i_{N-2}\in\N\\ i_1+\cdots+i_{N-2}\le q_N-1-j}}
	\binom{q_{N-1}}{j}\,
	\Biggl(\prod_{p=1}^{N-2}\binom{q_p}{i_p}\Biggr)\,
	\frac{1}{\bigl(q_N-1-j-(i_1+\cdots+i_{N-2})\bigr)!} \\
	&\qquad\qquad\qquad\cdot
	L(-1)^{q_N-1-j-(i_1+\cdots+i_{N-2})}
	\Bigl(u^{(1)}_{q_1-i_1}\cdots u^{(N-2)}_{q_{N-2}-i_{N-2}}\,
	u^{(N-1)}_{q_{N-1}-j}\,u^{(N)}\Bigr).
	\end{aligned}
	\end{equation}
\end{lem}

\begin{proof}
Use the creation property
$u^{(N)}_{-q_N}\mathbf1=L(-1)^{q_N-1}u^{(N)}/(q_N-1)!$.
Translation covariance gives $a_mL(-1)=L(-1)a_m+m a_{m-1}$;
induction on $t$ yields, for $a\in V$, $m\in\Z$, and $t\in\N$,
\begin{equation}\label{eq:mode-commute-L-1}
	a_m\,L(-1)^t
	=
	\sum_{j=0}^{t}
	\binom{t}{j}\,
	m(m-1)\cdots(m-j+1)\,
	L(-1)^{t-j}\,a_{m-j}.
	\end{equation}
where the empty product is $1$. Apply \eqref{eq:mode-commute-L-1} with
$a=u^{(N-1)}$, $m=q_{N-1}$, $t=q_N-1$, and use
\[
\frac{\binom{q_N-1}{j}\,q_{N-1}(q_{N-1}-1)\cdots(q_{N-1}-j+1)}{(q_N-1)!}
=\frac{\binom{q_{N-1}}j}{(q_N-1-j)!}.
\]
Left multiplication by the preceding modes proves \eqref{eq:N-iterate-step1}.
For \eqref{eq:N-iterate-step2}, put $t=q_N-1-j$ and apply
\eqref{eq:mode-commute-L-1} successively to the preceding modes, from right to left.
For $i_1+\cdots+i_{N-2}\le t$, the coefficient after division by $t!$ is
\[
\begin{aligned}
&\frac1{t!}\frac{t!}{i_1!\cdots i_{N-2}!\,(t-i_1-\cdots-i_{N-2})!}
\prod_{p=1}^{N-2}q_p(q_p-1)\cdots(q_p-i_p+1)\\
&\qquad=\frac{\prod_{p=1}^{N-2}\binom{q_p}{i_p}}{(t-i_1-\cdots-i_{N-2})!}.
\end{aligned}
\]
The remaining operator is
$L(-1)^{t-i_1-\cdots-i_{N-2}}u^{(1)}_{q_1-i_1}\cdots u^{(N-2)}_{q_{N-2}-i_{N-2}}$.
Multiplying by $\binom{q_{N-1}}j$ and substituting $t=q_N-1-j$
gives \eqref{eq:N-iterate-step2}.
\end{proof}

\begin{lem}[Expanding an $N$-iterate vacuum word against $R^{n}_{m,k}$]
	\label{lem:Niterate-word-Rmk-in-generators}
	Let $N\ge 2$. Let $u^{(1)},\dots,u^{(N)}\in V$ be homogeneous and let $z\in V$.
	Fix $n\in\N$ and $m,k\in\Z_{\ge0}$.
	Fix integers $q_1,\dots,q_{N-1}\in\Z$ and $q_N\in\N$ with $q_N\ge 1$.
	Set
	\[
	t:=q_N-1-j-(i_1+\cdots+i_{N-2})
	\qquad(\text{so }t\in\N\text{ on the summation range}).
	\]
	Then
	\begin{equation}\label{eq:Niterate-word-Rmk-in-generators}
	\begin{aligned}
	&\Bigl(u^{(1)}_{q_1}\cdots u^{(N-1)}_{q_{N-1}}\,u^{(N)}_{-q_N}\mathbf1\Bigr)\,R^{n}_{m,k}z \\
	&\quad=
	\sum_{j=0}^{q_N-1}
	\sum_{\substack{i_1,\dots,i_{N-2}\in\N\\ i_1+\cdots+i_{N-2}\le q_N-1-j}}
	\binom{q_{N-1}}{j}\,
	\Biggl(\prod_{p=1}^{N-2}\binom{q_p}{i_p}\Biggr)
	\\
	&\qquad\cdot
	\sum_{m'=m}^{m+t}
	\Biggl(
	\sum_{a=0}^{t-(m'-m)}
	(-1)^a\,
	\binom{n+k+\wt u^{(1)}+\cdots+\wt u^{(N)}-(q_1+\cdots+q_{N-1})+q_N-N}{a}
	\\
	&\qquad\qquad\qquad\qquad\cdot
	\binom{m+2n+t+1-a}{t-a}\,
	\binom{t-a}{m'-m}
	\Biggr)\;
	\Bigl(U_{j,\mathbf i}\,R^{n}_{m',k}z\Bigr),
	\end{aligned}
	\end{equation}
	where
	\[
	U_{j,\mathbf i}
	:=
	u^{(1)}_{q_1-i_1}\cdots u^{(N-2)}_{q_{N-2}-i_{N-2}}\,
	u^{(N-1)}_{q_{N-1}-j}\,u^{(N)}\in V.
	\]
\end{lem}

\begin{proof}
	Start from \eqref{eq:N-iterate-step2} and multiply on the right by $R^{n}_{m,k}z$:
	\[
	\Bigl(u^{(1)}_{q_1}\cdots u^{(N-1)}_{q_{N-1}}\,u^{(N)}_{-q_N}\mathbf1\Bigr)\,R^{n}_{m,k}z
	=
	\sum_{j,\mathbf i}
	\binom{q_{N-1}}{j}\Biggl(\prod_{p=1}^{N-2}\binom{q_p}{i_p}\Biggr)
	\frac{1}{t!}\;
	\Bigl(L(-1)^t U_{j,\mathbf i}\Bigr)\,R^{n}_{m,k}z,
	\]
	with $t=q_N-1-j-(i_1+\cdots+i_{N-2})$.

	Apply Lemma~\ref{lem:ur1-Rmk-change-of-basis} (with $r=t+1$ and $u=U_{j,\mathbf i}$)
	to rewrite the factor $\frac{1}{t!}\bigl(L(-1)^t U_{j,\mathbf i}\bigr)R^{n}_{m,k}z$ as a
	finite sum of $(U_{j,\mathbf i})R^{n}_{m',k}z$ for $m\le m'\le m+t$.
	This produces
\[
\begin{aligned}
\frac{1}{t!}\Bigl(L(-1)^t U_{j,\mathbf i}\Bigr)\,R^{n}_{m,k}z
&=
\sum_{m'=m}^{m+t}
\Biggl(
\sum_{a=0}^{t-(m'-m)}
(-1)^a\,
\binom{\wt U_{j,\mathbf i}+n+k+t}{a}\,
\binom{m+2n+t+1-a}{t-a} \\
&\qquad\qquad\cdot
\binom{t-a}{m'-m}
\Biggr)\,
\Bigl(U_{j,\mathbf i}\,R^{n}_{m',k}z\Bigr).
\end{aligned}
\]

	It remains to simplify $\wt U_{j,\mathbf i}+t$.
	Since $\wt(x_m y)=\wt x+\wt y-m-1$ for homogeneous $x,y$, iterating gives
	\[
	\begin{aligned}
	\wt U_{j,\mathbf i}
	&=
	\wt u^{(1)}+\cdots+\wt u^{(N)}
	-\sum_{p=1}^{N-2}(q_p-i_p+1)-(q_{N-1}-j+1) \\
	&=
	\wt u^{(1)}+\cdots+\wt u^{(N)}
	-\sum_{p=1}^{N-1}(q_p+1)+(j+i_1+\cdots+i_{N-2}).
	\end{aligned}
	\]
	Adding $t=q_N-1-j-(i_1+\cdots+i_{N-2})$ yields
	\[
	\wt U_{j,\mathbf i}+t
	=
	\wt u^{(1)}+\cdots+\wt u^{(N)}-(q_1+\cdots+q_{N-1})+q_N-N.
	\]
	Substituting this constant into the binomial $\binom{\wt U_{j,\mathbf i}+n+k+t}{a}$
	and inserting into the previous display gives \eqref{eq:Niterate-word-Rmk-in-generators}.
\end{proof}

\begin{lem}[Circle-subspace reduction for $u_i(v_{-r}\mathbf 1)$]
	\label{lem:On-ui-vk1-reduction}
	Fix $n\in\N$ and let $u,v\in V$ be homogeneous.  Fix $r\ge 1$ and $i\in\Z$.
	\begin{enumerate}
		\item[(a)] If $i\ge 0$, then
		\[
		O_n^\circ\!\bigl(u_i(v_{-r}\mathbf 1),V\bigr)
		\ \subseteq\
		\sum_{t\ge 0} O_n^\circ(u_t v,V).
		\]
		\item[(b)] If $i<0$, then
		\[
		O_n^\circ\!\bigl(u_i(v_{-r}\mathbf 1),V\bigr)
		\ \subseteq\
		\sum_{t\le -1} O_n^\circ(u_t v,V).
		\]
	\end{enumerate}
\end{lem}

\begin{proof}
	By Definition~\ref{def:Onu-subspace}, the subspace $O_n^\circ\!\bigl(u_i(v_{-r}\mathbf 1),V\bigr)$
	is spanned by the elements
	\[
	\bigl(u_i(v_{-r}\mathbf 1)\bigr)\,R^n_{m,k}z
	\qquad(z\in V,\ m,k\in\N,\ m\ge k).
	\]
	Fix such $z,m,k$.  Apply Lemma~\ref{lem:Niterate-word-Rmk-in-generators} with $N=2$,
	$u^{(1)}=u$, $u^{(2)}=v$, $q_1=i$, and $q_2=r$.
	Then \eqref{eq:Niterate-word-Rmk-in-generators} expresses
	\[
	\bigl(u_i(v_{-r}\mathbf 1)\bigr)\,R^n_{m,k}z
	\]
	as a (finite) $\C$--linear combination of terms
	\[
	(u_{i-j}v)\,R^n_{m',k}z
	\qquad(0\le j\le r-1,\ \ m\le m'\le m+(r-1-j)).
	\]
	Each such term lies in $O_n^\circ(u_{i-j}v,V)$ by Definition~\ref{def:Onu-subspace}.
	Hence
	\[
	O_n^\circ\!\bigl(u_i(v_{-r}\mathbf 1),V\bigr)\ \subseteq\ \sum_{j=0}^{r-1} O_n^\circ(u_{i-j}v,V).
	\]

	If $i\ge0$, the coefficient $\binom{i}{j}$ in
	\eqref{eq:Niterate-word-Rmk-in-generators} is zero for $j>i$. Thus the same expansion gives the sharper inclusion
	\[
	O_n^\circ\!\bigl(u_i(v_{-r}\mathbf1),V\bigr)
	\subseteq\sum_{j=0}^{\min\{i,r-1\}}O_n^\circ(u_{i-j}v,V)
	\subseteq\sum_{t\ge0}O_n^\circ(u_t v,V).
	\]
	If $i<0$, then $i-j\le -1$ for all $0\le j\le r-1$, so the right-hand side is contained in
	$\sum_{t\le -1}O_n^\circ(u_t v,V)$.
	This proves (a) and (b).
\end{proof}

\begin{cor}[First reduction for $O_n^\circ(N,S,V)$: $k=0$ and last mode $-1$]\label{cor:OnNSV-first-reduction}
	Fix $n\in\N$ and a finite homogeneous subset $S\subset V$, and let $N\ge1$.
	Recall $O_n^\circ(N,S,V)$ from Definition~\ref{def:OnN-SV}.
	With notation as in Definition~\ref{def:OnN-SV}, define the subset of vacuum words whose \emph{last} vacuum mode is equal to $-1$ by
	\[
	\mathcal P^{(-1)}_{\le N}(S)
	:=
	\Bigl\{
	w=s^{(1)}_{-m_1}\cdots s^{(r-1)}_{-m_{r-1}}\,s^{(r)}_{-1}\mathbf1
	\ \Big|\
	1\le r\le N,\ s^{(j)}\in S,\ m_1,\dots,m_{r-1}\in\Z
	\Bigr\}.
	\]
	Then
	\[
	O_n^\circ(N,S,V)
	=
	\Span_\C\Bigl\{
	w\,R^n_{m,0}z\ \Big|\ w\in\mathcal P^{(-1)}_{\le N}(S),\ m\in\Z_{\ge0},\ z\in V
	\Bigr\}.
	\]
	Equivalently: in Definition~\ref{def:OnN-SV} it is sufficient to restrict to $k=0$
	(using Lemma~\ref{Zhu-lemma}) and to restrict to words whose last factor is
	$s^{(r)}_{-1}\mathbf1$ (i.e.\ last vacuum mode $-1$), rather than $s^{(r)}_{-m_r}\mathbf1$ with
	$m_r\ge1$.
\end{cor}

\begin{proof}
	By Lemma~\ref{Zhu-lemma}, the span defining $O_n^\circ(N,S,V)$ is unchanged if we restrict
	to generators with $k=0$, so it suffices to consider elements
	\[
	\Bigl(s^{(1)}_{-m_1}\cdots s^{(r)}_{-m_r}\mathbf1\Bigr)\,R^n_{m,0}z,
	\qquad
	1\le r\le N,\ s^{(j)}\in S,\ m_1,\dots,m_{r-1}\in\Z,\ m_r\in\Z_{\ge1}.
	\]
	For $r=1$, Lemma~\ref{lem:ur1-Rmk-change-of-basis}, applied with $u=s^{(1)}$ and with $m_1$ in place of its descendant index (and $k=0$),
	reduces the generator to residues attached to $s^{(1)}=s^{(1)}_{-1}\mathbf1$.
	For $r\ge2$, apply Lemma~\ref{lem:Niterate-word-Rmk-in-generators} with
	$N=r$, $u^{(p)}=s^{(p)}$, $k=0$,
	\[
	q_p=-m_p\ \ (1\le p\le r-1),\qquad q_r=m_r,
	\]
	and the same $m\in\Z_{\ge0}$ and $z\in V$ as on the right of $R^n_{m,0}$.
	Lemma~\ref{lem:Niterate-word-Rmk-in-generators} expresses
	\[
	\Bigl(s^{(1)}_{-m_1}\cdots s^{(r)}_{-m_r}\mathbf1\Bigr)\,R^n_{m,0}z
	\]
	as a finite $\C$--linear combination of terms of the form
	\(
	U_{j,\mathbf i}\,R^n_{m',0}z
	\)
	(with various $m'\ge m$), where
	\[
	U_{j,\mathbf i}
	=
	s^{(1)}_{q_1-i_1}\cdots s^{(r-2)}_{q_{r-2}-i_{r-2}}\,
	s^{(r-1)}_{q_{r-1}-j}\,s^{(r)}.
	\]
	Thus each term lies in the span of generators $w\,R^n_{m',0}z$ with
	$w\in\mathcal P^{(-1)}_{\le N}(S)$.
	Conversely, each generator $w\,R^n_{m,0}z$ with
	$w\in\mathcal P^{(-1)}_{\le N}(S)$ is among the defining generators of
	$O_n^\circ(N,S,V)$.  Hence the right-hand span is contained in
	$O_n^\circ(N,S,V)$, and the claimed equality of spans follows.
\end{proof}

\begin{rem}[Mode indices in singular and regular iterate expansions]\label{rem:YL-safe-vs-regular}
	A key point in the computation of $O_n^\circ(2,\{u,v\},V)$ is that the circle-type relations
have truncated representatives involving only modes with indices at least $-2n-1$.
We distinguish expansions that preserve this lower bound from those that introduce
mode indices less than $-2n-1$.

\smallskip
\noindent{\rm(1) Singular iterates \((i\ge0)\).}
For \(i\ge0\), the standard identity
\begin{equation}\label{eq:LL-iterate-to-comm}
(u_i v)_j
=
\sum_{p=0}^{i}
(-1)^p\binom{i}{p}\bigl[\,u_{i-p},v_{j+p}\,\bigr]
\end{equation}
expresses each coefficient \((u_i v)_j\) as a finite sum of commutators.  If
\(j\ge-2n-1\), then \(i-p\ge0\ge-2n-1\) and \(j+p\ge j\ge-2n-1\).  Thus
singular iterate expansions involve only modes with indices at least $-2n-1$.

	\smallskip
	\noindent{\rm(2) Regular iterates ($r\ge1$).}
	By contrast, for $r\ge1$ the modes $(u_{-r}v)_j$ (even with $j\ge -2n-1$) generally expand
	into quadratic monomials with mode indices less than $-2n-1$.  For this reason we do \emph{not} expand $Y_L(u_{-r}v,x)$ directly.
	Instead we use the one-step replacement
	\begin{equation}\notag
	Y(u_{-r}v,x)
	\;\equiv_{U}\;
	\sum_{i\ge0} F^{+}_{i,r,u}(x,x)\,Y_L(u_i v,x)
	+
	Y_L(v,x)\,Y_L(u_{-r}\mathbf{1},x),
	\end{equation}
from Lemma~\ref{lem:one-step-general} (see \eqref{eq:one-step-YL-U}), which expresses the regular iterate in terms of
singular iterate fields together with a vacuum-descendant field.
After expanding the right-hand side, every quadratic monomial has mode indices
at least $-2n-1$, leaving the regular quadratic relations to be analyzed.
\end{rem}
\begin{lem}[Quadratic reduction for $O_n^\circ(2,\{u,v\},V)$]\label{lem:On2-cutoff-reduction}
	Fix $n\in\N$, let $V$ be a vertex operator algebra, and let
	$u,v\in V$ be homogeneous.  Set $S=\{u,v\}$.  Put
	\[
	\Pi=\{(u,u),(u,v),(v,v)\}.
	\]
	Let $U$ be the linear subspace generated by the circle-type relations coming from
	$u$, $v$, and the singular iterates attached to the three pair types in $\Pi$:
	\[
	U
	:=
	O_n^\circ(u,V)+O_n^\circ(v,V)
	+
	\sum_{(s,t)\in\Pi}\ \sum_{i\ge0}O_n^\circ(s_i t,V).
	\]
	Then
	\[
	O_n^\circ(2,\{u,v\},V)^+=U.
	\]
	Moreover, the quotient $O_n^\circ(2,\{u,v\},V)/U$ is spanned by the classes of
	relations
	\[
	(s_{-r}t)R^n_{m,0}z,
	\qquad
	(s,t)\in\Pi,\quad r\ge1,\quad m\ge0,\quad z\in V.
	\]
	In particular, after imposing the relations in $U$, the only remaining
	length-$2$ circle-type relations are the true regular quadratic iterates for the
	three ordered pair types
	\[
	u_{-r}u,\qquad u_{-r}v,\qquad v_{-r}v
	\qquad (r\ge1).
	\]
The reversed cross-pair \(v_q u\), \(q\in\Z\), contributes no additional pair
type; by skew-symmetry it is reduced to singular iterates \(u_i v\), \(i\ge0\),
and regular iterates \(u_{-r'}v\), \(r'\ge1\), up to \(L(-1)\)-descendant
reductions.
\end{lem}

\begin{proof}
	\emph{The equality $O_n^\circ(2,\{u,v\},V)^+=U$.}  We use that $R^n_{m,k}$ is linear in its left
	argument on each weight space, so $O_n^\circ(\sum_\rho c_\rho,V)\subseteq\sum_\rho O_n^\circ(c_\rho,V)$
	for homogeneous $c_\rho$ of a common weight.  By Definition~\ref{def:OnN-SV}, the plus subspace
	is spanned by the relations $wR^n_{m,k}z$ with either $w=s_{-m_1}\mathbf1$, $s\in\{u,v\}$,
	$m_1\ge1$, or $w=s_it_{-k'}\mathbf1$, $s,t\in\{u,v\}$, $i\ge0$, $k'\ge1$.  In the first case
	$s_{-m_1}\mathbf1=\frac{1}{(m_1-1)!}L(-1)^{m_1-1}s$, and Lemma~\ref{lem:Onu-Lminus1-inclusion},
	applied $m_1-1$ times, gives $O_n^\circ(w,V)\subseteq O_n^\circ(s,V)\subseteq U$.  In the
	second case Lemma~\ref{lem:On-ui-vk1-reduction}\textup{(a)} gives
	$O_n^\circ(w,V)\subseteq\sum_{q\ge0}O_n^\circ(s_qt,V)$, which lies in $U$ when $(s,t)\in\Pi$.
	When $(s,t)=(v,u)$, the skew-symmetry formula for $v_qu$ displayed below, with $q\ge0$,
	writes $v_qu$ as a finite sum of the homogeneous vectors
	$\frac{(-1)^{q+j+1}}{j!}L(-1)^j(u_{q+j}v)$, all of weight $\wt u+\wt v-q-1$; with
	Lemma~\ref{lem:Onu-Lminus1-inclusion} this gives
	$O_n^\circ(v_qu,V)\subseteq\sum_{j\ge0}O_n^\circ(u_{q+j}v,V)\subseteq U$.  Hence the plus
	subspace is contained in $U$.  Conversely, $O_n^\circ(s,V)=O_n^\circ(s_{-1}\mathbf1,V)$ and
	$O_n^\circ(s_it,V)=O_n^\circ(s_it_{-1}\mathbf1,V)$ with $i\ge0$ are spanned by generators of the
	plus subspace, so $U$ is contained in it.

	\smallskip
	\emph{The quotient.}
	By Corollary~\ref{cor:OnNSV-first-reduction}, $O_n^\circ(2,\{u,v\},V)$ is
	spanned by relations
	\[
	wR^n_{m,0}z,
	\qquad
	m\ge0,\quad z\in V,
	\]
	where $w$ ranges over vacuum words of length at most $2$ in modes of $u$ and
	$v$ with last vacuum mode $-1$.

	Words of length $1$ give precisely $O_n^\circ(u,V)$ and
	$O_n^\circ(v,V)$, hence lie in $U$.

For words of length \(2\), the non-reversed pair types are represented by
\[
s_{-r}t,
\qquad
(s,t)\in\Pi,\quad r\in\Z.
\]
If \(r\le0\), then \(s_{-r}t=s_i t\) with \(i=-r\ge0\), so the corresponding
relations lie in \(U\) by definition.  Thus, modulo \(U\), only the regular
iterates \(s_{-r}t\) with \(r\ge1\) remain for the three pair types in \(\Pi\).

It remains to treat the reversed cross-pair.  For \(q\in\Z\), skew-symmetry gives
\[
Y(v,x)u=e^{xL(-1)}Y(u,-x)v,
\]
and extracting the coefficient of \(x^{-q-1}\) gives
\[
v_q u
=
\sum_{j\ge0}
\frac{(-1)^{q+j+1}}{j!}\,
L(-1)^j\bigl(u_{q+j}v\bigr).
\]
Only finitely many terms are nonzero by truncation.  If \(q\ge0\), then every
\(u_{q+j}v\) is a singular iterate, so the corresponding circle-type relations
lie in \(U\), after applying Lemma~\ref{lem:Onu-Lminus1-inclusion} repeatedly
to remove the \(L(-1)\)-descendants.  If \(q<0\), say \(q=-r\) with \(r\ge1\),
then the terms with \(j\ge r\) are again \(L(-1)\)-descendants of singular
iterates and hence lie in \(U\), while the terms with \(j<r\) are
\(L(-1)\)-descendants of regular iterates \(u_{-(r-j)}v\).  By
Lemma~\ref{lem:ur1-Rmk-change-of-basis}, applied repeatedly, these
\(L(-1)\)-powers may be removed from the left input of the residue relation, at
the expense of replacing \(R^n_{m,0}\) by a finite linear combination of
operators \(R^n_{m',0}\).  Hence, modulo \(U\), every relation coming from the
reversed cross-pair is either zero or a finite linear combination of relations
coming from \(u_{-r'}v\) with \(r'\ge1\).

Therefore, modulo \(U\), the length-\(2\) quotient is spanned by the regular
relations
\[
(s_{-r}t)R^n_{m,0}z,
\qquad
(s,t)\in\Pi,\quad r\ge1,\quad m\ge0,\quad z\in V,
\]
as claimed.
\end{proof}

\section{Examples and Outlook}\label{sec:examples}
\subsection{Length-one and Virasoro quadratic examples}
We now illustrate the finite-length formalism in the first two word lengths.  The first example is the length-one case and recovers the subspaces \(O_n^\circ(u,V)\) already studied above.  The second example records the Virasoro length-two window and spells out the \(n=1\) and \(n=2\) relations explicitly.
\begin{Exa}[Length one]\label{exa:length-one-Onu}
For homogeneous $u\in V$, Lemma~\ref{lem:ur1-Rmk-change-of-basis}
reduces each residue attached to $u_{-r}\mathbf1$ to residues attached
to $u$. Conversely $u=u_{-1}\mathbf1$. Reducing $k$ to zero by
Lemma~\ref{Zhu-lemma} therefore gives
\begin{equation}\label{eq:length-one-circle-space}
O_n^\circ(1,\{u\},V)
=\Span_\C\{uR^n_{m,0}z:\ m\ge0,\ z\in V\}
=O_n^\circ(u,V).
\end{equation}
\end{Exa}

\begin{Exa}[Virasoro quadratic windows]\label{exa:virasoro-quadratic-window}
	Let \(V_c\subset V\) be the Virasoro vertex operator subalgebra with conformal vector
	\(\omega\), and recall
	\[
	Y(\omega,x)=\sum_{q\in\Z}L(q)x^{-q-2},
	\qquad
	\omega_j=L(j-1),
	\qquad
	\wt\omega=2.
	\]
	Set \(S=\{\omega\}\).  We consider the length-two quotient
	\[
	O_n^\circ(2,\{\omega\},V)/
	O_n^\circ(2,\{\omega\},V)^+.
	\]

	The singular Virasoro iterates are
	\[
	\omega_0\omega=L(-1)\omega,\qquad
	\omega_1\omega=2\omega,\qquad
	\omega_2\omega=0,\qquad
	\omega_3\omega=\frac c2\mathbf1,
	\qquad
	\omega_i\omega=0\quad(i\ge4).
	\]
	Hence the plus part is generated by the length-one relations coming from \(\omega\),
	together with the relations coming from these singular iterates.  The
	\(L(-1)\omega\) term is already contained in \(O_n^\circ(\omega,V)\) by
	Lemma~\ref{lem:Onu-Lminus1-inclusion}, and the \(\omega_1\omega=2\omega\)
	term gives no new subspace.  The vacuum term also contributes no nonzero
	relations.  Indeed, for \(z\in V\) and \(m\ge k\ge0\),
	\[
	\mathbf1\,R^n_{m,k}z
	=
	\Res_x\frac{(1+x)^{n+k}}{x^{m+2n+2}}z
	=
	[x^{m+2n+1}](1+x)^{n+k}z
	=
	0,
	\]
	because
	\[
	m+2n+1\ge k+2n+1>n+k.
	\]
	Therefore, at the induced-subspace level,
	\[
	O_n^\circ(2,\{\omega\},V)^+
	=
	O_n^\circ(1,\{\omega\},V)=	O_n^\circ(\omega,V).
	\]

By Lemma~\ref{lem:On2-cutoff-reduction}, after quotienting by
\(O_n^\circ(\omega,V)\), the induced length-two quotient is represented by
the true regular quadratic Virasoro iterates
\[
(\omega_{-r}\omega)R^n_{m,0}z,
\qquad
r\ge1,\quad m\ge0,\quad z\in V.
\]
At the operator level, the finite spanning window is most naturally expressed
using the two boundary strips \(k=0\) and \(k=n\).
	The point of the \(Y_L\)-formalism is that these regular iterate fields are not
	expanded directly.  Instead, Lemma~\ref{lem:one-step-general}, applied with
	\(u=v=\omega\), gives the congruence
	\[
	Y(\omega_{-r}\omega,x)
	\equiv_{\,O_n^\circ(\omega,V)\,}
	\widehat Y^{(n)}_{r,\omega}(x),
	\]
	where
	\[
	\widehat Y^{(n)}_{r,\omega}(x)
	:=
	\sum_{i\ge0}F^+_{i,r,\omega}(x,x)Y_L(\omega_i\omega,x)
	+
	Y_L(\omega,x)Y_L(\omega_{-r}\mathbf1,x).
	\]
	Using the Virasoro singular iterates, this becomes
	\[
	\begin{aligned}
	\widehat Y^{(n)}_{r,\omega}(x)
	={}&
	F^+_{0,r,\omega}(x,x)Y_L(L(-1)\omega,x)
	+
	2F^+_{1,r,\omega}(x,x)Y_L(\omega,x)  \\
	&+
	\frac c2F^+_{3,r,\omega}(x,x)\Id
	+
	Y_L(\omega,x)Y_L(\omega_{-r}\mathbf1,x).
	\end{aligned}
	\]
Here, by Convention~\ref{conv:descendant-derivative}, the final factor is the derivative field
	\[
Y_L(\omega_{-r}\mathbf1,x)
=
\frac{1}{(r-1)!}\partial_x^{r-1}Y_L(\omega,x).
	\]
	Thus the replacement residue associated to \(\omega_{-r}\omega\) is
	\[
	(\omega_{-r}\omega)R^{n,\widehat Y}_{m,k}z
	:=
	\Res_x
	\frac{(1+x)^{r+3+n+k}}{x^{m+2n+2}}\,
	\widehat Y^{(n)}_{r,\omega}(x)z.
	\]
	This residue is congruent to the original circle-type relation
	\[
	(\omega_{-r}\omega)R^n_{m,k}z
	\]
	modulo \(O_n^\circ(\omega,V)\), by Proposition~\ref{prop:Rmk-replacement-in-On}.
	Therefore the quadratic relations may be computed using
	\(\widehat Y^{(n)}_{r,\omega}(x)\), which only involves the
		field \(Y_L(\omega,x)\), its vacuum descendants, and singular Virasoro data.
		Thus the problem is reduced to finding a finite set of triples \((r,m,k)\)
		whose residues, applied to arbitrary inputs, span the quadratic
relations modulo \(O_n^\circ(\omega,V)\).

The displayed representatives below are computed here. It can be shown,
using the equal-input reduction of \cite{Part2}, that
modulo $O_n^\circ(\omega,V)$ all quadratic residues reduce to
\begin{equation}\label{eq:virasoro-quadratic-spanning-window}
(\omega_{-r}\omega)R^n_{m,k}z,
\qquad n\le m\le2n-1,\quad 1\le r\le2n-m,\quad r\text{ odd},\quad k\in\{0,n\}.
\end{equation}
The operators are applied to arbitrary $z\in V$.
For every $r\ge1$ and $m\ge k\ge0$, it can be shown using \cite{Part2} that the singular replacement term
\[
\Res_x
\frac{(1+x)^{r+3+n+k}}{x^{m+2n+2}}
\sum_{i\ge0}F^+_{i,r,\omega}(x,x)Y_L(\omega_i\omega,x)
\]
is zero in $V$.  Thus the
representatives below come from the product term
\[
\Res_x
\frac{(1+x)^{r+3+n+k}}{x^{m+2n+2}}
Y_L(\omega,x)Y_L(\omega_{-r}\mathbf1,x),
\]
in accordance with \eqref{eq:conj-first-term-residue}.
Expanding the two $p$-series by \eqref{eq:pj-def} and applying
the residue kernel \eqref{eq:Rmk-def} gives the following explicit
operator congruences:
\[
\begin{aligned}
(\omega_{-r}\omega)R^n_{m,0}(\cdot)
&\equiv_{O_n^\circ(\omega,V)}
\sum_{\alpha,\beta\ge -2n-1}
\sum_{d=-\alpha-1}^{2n}
\sum_{e=-\beta-r}^{2n}
\sum_{t=0}^{\min(r-1,\,2n-e)}
\binom{-n}{m+2n+1-d-e}
 \\
&{}\cdot \binom{n+2}{d+\alpha+1}
\binom{n+2}{e+\beta+1+t}
\binom{e+t}{t}
\binom{e-n-2}{r-1-t}
L(\alpha-1)L(\beta-1),
\end{aligned}
\]
and
\[
\begin{aligned}
(\omega_{-r}\omega)R^n_{m,n}(\cdot)
&\equiv_{O_n^\circ(\omega,V)}
\sum_{\alpha,\beta\ge -2n-1}
\sum_{d=-\alpha-1}^{2n}
\sum_{e=-\beta-r}^{2n}
\sum_{t=0}^{\min(r-1,\,2n-e)}
\binom{0}{m+2n+1-d-e}
\\
&{}\cdot
\binom{n+2}{d+\alpha+1}
\binom{n+2}{e+\beta+1+t}
\binom{e+t}{t}
\binom{e-n-2}{r-1-t}
L(\alpha-1)L(\beta-1).
\end{aligned}
\]
The sums are interpreted after application to each input $z\in V$: lower truncation first makes the rightmost-mode sum finite, and then makes the leftmost-mode sum finite for each resulting vector. The two odd-indexed windows contain
$2\lfloor(n+1)^2/4\rfloor$ operators in total.

For \(n=1\), the two spanning relations are
\[
\begin{aligned}
(\omega_{-1}\omega)R^1_{1,0}(\cdot)
\equiv_{O_1^\circ(\omega,V)}
&L(-4)\bigl(L(-4)+2L(-3)+L(-2)\bigr) \\
&{}+L(-3)\bigl(2L(-4)+4L(-3)+2L(-2)\bigr) \\
&{}+L(-2)\bigl(L(-4)+2L(-3)+L(-2)\bigr),
\end{aligned}
\]
and
\[
\begin{aligned}
(\omega_{-1}\omega)R^1_{1,1}(\cdot)
\equiv_{O_1^\circ(\omega,V)}
&L(-4)\bigl(L(-4)+3L(-3)+3L(-2)+L(-1)\bigr) \\
&{}+L(-3)\bigl(3L(-4)+9L(-3)+9L(-2)+3L(-1)\bigr) \\
&{}+L(-2)\bigl(3L(-4)+9L(-3)+9L(-2)+3L(-1)\bigr) \\
&{}+L(-1)\bigl(L(-4)+3L(-3)+3L(-2)+L(-1)\bigr).
\end{aligned}
\]

For \(n=2\), the four spanning relations are
\[
\begin{aligned}
(\omega_{-1}\omega)R^2_{2,0}(\cdot)
\equiv_{O_2^\circ(\omega,V)}
&L(-6)\bigl(L(-5)+2L(-4)+L(-3)\bigr) \\
&{}+L(-5)\bigl(L(-6)+6L(-5)+9L(-4)+4L(-3)\bigr) \\
&{}+L(-4)\bigl(2L(-6)+9L(-5)+12L(-4)+5L(-3)\bigr) \\
&{}+L(-3)\bigl(L(-6)+4L(-5)+5L(-4)+2L(-3)\bigr),
\end{aligned}
\]

\[
\begin{aligned}
(\omega_{-1}\omega)R^2_{3,0}(\cdot)
\equiv_{O_2^\circ(\omega,V)}
&L(-6)\bigl(L(-6)+2L(-5)+L(-4)\bigr) \\
&{}+L(-5)\bigl(2L(-6)+3L(-5)-L(-3)\bigr) \\
&{}+L(-4)\bigl(L(-6)-3L(-4)-2L(-3)\bigr) \\
&{}+L(-3)\bigl(-L(-5)-2L(-4)-L(-3)\bigr),
\end{aligned}
\]
\[
\begin{aligned}
(\omega_{-1}\omega)R^2_{2,2}(\cdot)
&\equiv_{O_2^\circ(\omega,V)}
L(-6)\bigl(L(-5)+4L(-4)+6L(-3)+4L(-2)+L(-1)\bigr) \\
&{}+L(-5)\bigl(L(-6)+8L(-5)+22L(-4)+28L(-3)+17L(-2)+4L(-1)\bigr) \\
&{}+L(-4)\bigl(4L(-6)+22L(-5)+48L(-4)+52L(-3)+28L(-2)+6L(-1)\bigr) \\
&{}+L(-3)\bigl(6L(-6)+28L(-5)+52L(-4)+48L(-3)+22L(-2)+4L(-1)\bigr) \\
&{}+L(-2)\bigl(4L(-6)+17L(-5)+28L(-4)+22L(-3)+8L(-2)+L(-1)\bigr) \\
&{}+L(-1)\bigl(L(-6)+4L(-5)+6L(-4)+4L(-3)+L(-2)\bigr),
\end{aligned}
\]

and
\[
\begin{aligned}
(\omega_{-1}\omega)R^2_{3,2}(\cdot)
&\equiv_{O_2^\circ(\omega,V)}
L(-6)\bigl(L(-6)+4L(-5)+6L(-4)+4L(-3)+L(-2)\bigr) \\
&{}+L(-5)\bigl(4L(-6)+16L(-5)+24L(-4)+16L(-3)+4L(-2)\bigr) \\
&{}+L(-4)\bigl(6L(-6)+24L(-5)+36L(-4)+24L(-3)+6L(-2)\bigr) \\
&{}+L(-3)\bigl(4L(-6)+16L(-5)+24L(-4)+16L(-3)+4L(-2)\bigr) \\
&{}+L(-2)\bigl(L(-6)+4L(-5)+6L(-4)+4L(-3)+L(-2)\bigr).
\end{aligned}
\]
Thus, at level $n=2$, these four operators span the quadratic relations
modulo $O_2^\circ(\omega,V)$ by \eqref{eq:virasoro-quadratic-spanning-window}.
Their displayed representatives are computed from the product part of
$\widehat Y^{(2)}_{1,\omega}(x)$.
\end{Exa}

\subsection{Quadratic counts and a finite-list preview}
\label{subsec:finite-list-preview}
For each distinct pair $(s,t)$, the residues
$(s_{-r}t)R^n_{m,k}$ with $n\le m\le2n-1$, $1\le r\le2n-m$,
and $k\in\{0,n\}$ suffice \cite{Part2}. Each of the two triangular windows has
\begin{equation}\label{eq:quadratic-window-count}
\sum_{m=n}^{2n-1}(2n-m)=\binom{n+1}{2}
\end{equation}
operators. For an equal-input pair only odd $r$ are needed, giving
$\lfloor(n+1)^2/4\rfloor$ operators per window.
Each operator is applied to all $z\in V$.

At level $n=2$, a proposed list under the auxiliary-degree hypothesis
can also be written for an arbitrary finite $S$. Fix an ordering of $S$.
The additional cubic operators are
\begin{equation}\label{eq:auxiliary-cubic-list-n2}
(a_{-1}b_{-1}c)R^2_{4,k}(\cdot),\qquad a\le b\le c,\quad a,b,c\in S,
\quad k\in\{0,4\}.
\end{equation}
The proposed refinement is
\begin{equation}\label{eq:auxiliary-complete-list-n2}
O_2^\circ(V)=O_2^\circ(2,S,V)
+\Span_\C\{T(z):\ T\text{ in }\eqref{eq:auxiliary-cubic-list-n2},\ z\in V\}.
\end{equation}
There are at most $2\binom{|S|+2}{3}$ additional cubic operators.
This refinement requires cubic reductions to be given in subsequent work.

A concrete refinement of Conjecture~\ref{conj:finite-length-zhu-relation}
is available when $V$ is $\N$-graded and strongly generated by one
homogeneous vector $u$. The proposed generating list consists of
\begin{equation}\label{eq:single-generator-length-one-list}
uR^n_{m,0}(\cdot),\qquad m\ge0,
\end{equation}
together with the finite quadratic list
\begin{equation}\label{eq:single-generator-quadratic-list}
(u_{-r}u)R^n_{m,k}(\cdot),\qquad
n\le m\le2n-1,\quad 1\le r\le2n-m,\quad r\text{ odd},\quad k\in\{0,n\}.
\end{equation}
At levels $n\le3$, the proposed additional higher-length operators are
\begin{align}
n=1:&\quad\text{none},\label{eq:single-generator-higher-list-n1}\\
n=2:&\quad(u_{-1}u_{-1}u)R^2_{4,k}(\cdot),\quad k\in\{0,4\},
\label{eq:single-generator-higher-list-n2}\\
n=3:&\quad(u_{-1}u_{-1}u)R^3_{6,k}(\cdot),\quad k\in\{0,6\},\notag\\
&\quad(u_{-1}u_{-1}u_{-1}u)R^3_{9,k}(\cdot),\quad k\in\{0,9\}.
\label{eq:single-generator-higher-list-n3}
\end{align}
The proposed assertion is that the values of these operators, on all
$z\in V$, span $O_n^\circ(V)$. The quadratic reductions are given
in \cite{Part2}. The passage from the finite-length conjecture to the complete
lists, including the cubic and quartic operators above, requires the
higher-length reductions to be given in subsequent work.

\section*{Acknowledgments}
The author thanks Katrina Barron for discussions.

\section*{Code availability}
Exact-arithmetic Python scripts (standard library only) that check selected identities and examples of this paper
over finite ranges are available as ancillary files with the arXiv version; see \texttt{anc/README.txt}.
They test the change-of-generators formulas of Section~\ref{zhu-algebra-definition-section}
(Lemmas~\ref{lem:ur1-Rmk-change-of-basis} and~\ref{lem:Onu-Lminus1-inclusion} and
Corollary~\ref{lem:circle-Rk0-change-of-basis}), the series $p_{j,u}$ and their derivatives
(Lemmas~\ref{lem:p-decomposition} and~\ref{lem:p-higher-derivatives} and
Example~\ref{exa:YL-omega-n2}), the regular-part formula of Proposition~\ref{prop:AB-specialization},
the diagonal kernel of Lemma~\ref{lem:Cium-YL}, and the Virasoro lists of
Example~\ref{exa:virasoro-quadratic-window}.
The proofs do not depend on these computations.

\section*{Declaration of generative AI and AI-assisted technologies in the manuscript preparation process}
During the preparation of this work the author used ChatGPT (OpenAI) and Claude (Anthropic) in order to assist with drafting and revising the manuscript, including its mathematical exposition, and to check identities by exact computer calculation. After using these tools, the author reviewed and edited the content as needed and takes full responsibility for the content of the published article.
\begingroup
\raggedright

\endgroup
\end{document}